\documentclass[12pt]{article}
\usepackage[letterpaper,margin=0.85in]{geometry}

\usepackage[british]{babel}
\usepackage[latin9]{inputenc}
\usepackage{amssymb}
\usepackage{amsthm}
\usepackage{xcolor}
\usepackage{blindtext}
\usepackage{amsmath, amsfonts}
\usepackage{empheq}
\usepackage{mdframed}

\DeclareMathOperator{\dist}{dist}

\DeclareMathOperator*{\argmin}{arg\,min}

\usepackage{txfonts}
\usepackage{enumerate}
\numberwithin{equation}{section}
\usepackage{tabulary}
\usepackage{bbm}
\usepackage{verbatim}

\usepackage{lmodern}
\usepackage{paralist}
\usepackage[activate]{pdfcprot}
  \usepackage{paralist}
  \usepackage{graphics} 
  \usepackage{epsfig} 
\usepackage{graphicx}
\usepackage{epic}
\usepackage[Option]{overpic}
\usepackage{epstopdf}
\usepackage[font={small,it}]{caption}
\usepackage{subcaption}
\usepackage{float}
\usepackage{slashed}
\usepackage{mathrsfs}
\usepackage{oldgerm}
\usepackage{dsfont}
\usepackage{setspace}
\usepackage{mathtools}
\usepackage[all]{xy}

\newcommand{\rmd}{\mathrm{d}}
\newcommand{\rmD}{\mathrm{D}}
\renewcommand{\epsilon}{\varepsilon}

\theoremstyle{plain}
\newtheorem{theorem}{Theorem}[section]

\newtheorem{lemma}[theorem]{Lemma}
\newtheorem{proposition}[theorem]{Proposition}

\theoremstyle{definition}

\theoremstyle{definition}
\newtheorem{definition and lemma}[theorem]{Definition and Lemma}

\theoremstyle{remark}

\newcommand{\tostar}{\overset{*}{\lower0.5em\hbox{$\smash{\scriptscriptstyle\rightharpoonup}$}}}

\def\txta{{\textnormal{a}}}

\def\txte{{\textnormal{e}}}

\def\txtr{{\textnormal{r}}}

\newcommand{\be}{\begin{equation}}
\newcommand{\ee}{\end{equation}}
\newcommand{\benn}{\begin{equation*}}
\newcommand{\eenn}{\end{equation*}}
\newcommand{\bea}{\begin{eqnarray}}
\newcommand{\eea}{\end{eqnarray}}
\newcommand{\beann}{\begin{eqnarray*}}
\newcommand{\eeann}{\end{eqnarray*}}

\numberwithin{equation}{section}

\title{Desingularization of the transcritical bifurcation in discrete time}

\author{Maximilian Engel}  
\date{\today} 

\begin{document}
\author{Maximilian Engel\thanks{Zentrum Mathematik der TU M\"{u}nchen,
Boltzmannstr. 3, D-85748 Garching bei M\"{u}nchen}~ and Christian Kuehn\footnotemark[1]}
 
\title{Discretized Fast-Slow Systems \\ 
near Transcritical Singularities}

\maketitle

\begin{abstract}
We extend slow manifolds near a transcritical singularity in a fast-slow 
system given by the explicit Euler discretization of the corresponding 
continuous-time normal form. The analysis uses the blow-up method and 
direct trajectory-based estimates. We prove that the qualitative behaviour 
is preserved by a time-discretization with sufficiently small step size.
This step size is fully quantified relative to the time scale separation. Our
proof also yields the continuous-time results as a special case and provides
more detailed calculations in the classical (or scaling) chart.
\end{abstract}


\section{Introduction}
We study the dynamics of the two-dimensional quadratic map
\begin{equation} 
\label{map_intro}
p: \begin{pmatrix}
x \\ y 
\end{pmatrix}
\mapsto \begin{pmatrix}
\tilde{x}  \\ \tilde y 
\end{pmatrix}
= \begin{pmatrix}
x + h(x^2 - y^2 + \lambda \epsilon) \\
y + \epsilon h
\end{pmatrix}
\end{equation}
for $h, \epsilon > 0$. We interpret $\epsilon$ as a small time scale 
sepration parameter between the fast variable $x$ and the slow variable $y$. 
The parameter $h$ can be viewed as the step size for the explicit Euler 
discretization of the corresponding ordinary differential equation (ODE)
\begin{align} 
\label{ODE_intro}
\begin{array}{r@{\;\,=\;\,}l}
x' & x^2 - y^2 + \lambda \epsilon \,, \\
y' & \epsilon\,,
\end{array}
\end{align}
which represents the normal form of a fast-slow system exhibiting 
a \emph{transcritical singularity} at the origin. The term transcritical 
refers to the fact that, if $y$ is seen as a bifurcation parameter for 
the flow in the $x$-variable, a transcritical bifurcation occurs at the 
origin $(x,y)=(0,0)$. The origin is singular since hyperbolicity of the 
dynamics breaks down at this point. The same holds for the 
map~\eqref{map_intro}. 

In the case of~\eqref{ODE_intro}, Krupa and Szmolyan~\cite{ks2001/2} analyze 
the dynamics around the origin by using the \emph{blow-up method} for vector 
fields with singularities. The key idea to use the blow-up 
method~\cite{Du78,Du93} for fast-slow systems goes back to Dumortier and 
Roussarie~\cite{DuRo96}. They observed that this technique may convert non-hyperbolic 
singularities at which fast and slow directions interact into partially hyperbolic 
problems. The method inserts a suitable manifold, e.g.~a sphere, at the singularity 
and describes the extension of hyperbolic objects through a neighbourhood of the 
singularity via the partially hyperbolic dynamics on this manifold; see 
e.g.~\cite[Chapter~7]{ku2015} for an introduction and~\cite{DeMaesschalckDumortier7,
DeMaesschalckDumortier4,DeMaesschalckWechselberger,GucwaSzmolyan,ks2011,KuehnUM,KuehnHyp} 
for a list of a few, yet by no means exhaustive, list of different applications to 
planar fast-slow systems.

Krupa and Szmolyan also used the blow-up method for normal form of \emph{fold 
singularities} in fast-slow systems~\cite{ks2011}. For this case Nipp and 
Stoffer~\cite{ns2013} transform the blow-up technique to the corresponding 
explicit Runge-Kutta, in particular Euler, discretization and prove the extension 
of slow manifolds for the discrete time system around the singularity. They treat 
the discretized dynamics in vicinity to the fold singularity as an application of 
a more general existence theory for invariant manifolds they develop in~\cite{ns2013}.

In this paper, we use instead a direct approach to analyze, how trajectories induced 
by~\eqref{map_intro} pass the singularity at the origin. This approach allows us to 
obtain pathwise control over the transitions in and within the different blow-up
charts. The singularity is blown up to a sphere on which trajectories can be described 
directly via the map. This leads to the main result of the paper, 
Theorem~\ref{transcritical_discrete}, which is the discrete-time extension 
of~\cite[Theorem 2.1]{ks2001/2}. In this context, we only impose that $h$ is bounded by $\epsilon$ and prove that there is no further restriction on 
the step size. Our theorem states explicitly, how for 
the cases $\lambda < 1$ and $\lambda > 1$ in~\eqref{map_intro} attracting slow manifolds 
extend beyond the singularity at the origin. It gives estimates on the contraction rates 
of neighbourhoods of the manifolds. The case $\lambda=1$ corresponds with the problem of 
\emph{canard solutions}~\cite{BenoitCallotDienerDiener,DuRo96,ks2001/3} and will be dealt 
with in future work. It should be noted that, by letting $h \to 0$, our proof of 
Theorem~\ref{transcritical_discrete} can also be seen as a different way of 
proving \cite[Theorem 2.1]{ks2001/2} and our proof makes the results~\cite{ks2001/2} for 
the scaling chart more explicit. Additionally, the blow-up method provides the insight that only in one chart around the singularity the preservation of stability behaviour is bound to the stability criteria of the Euler discretisation derived from the Dahlquist test equation while in the other charts there is no such restriction.

This work lays the foundation of a broader effort to apply the blow-up method, which has 
so far mainly been used for flows, to fast-slow dynamical systems induced by maps. First, 
it is insightful to look at key examples that can be compared to continuous-time 
analogues, as in the case of the transcritical singularity. In the future, also 
multiscale discrete-time problems, which have no correspondence to fast-slow flows, will 
be considered.

The remainder of the paper is structured as follows. In Section~\ref{continuoustime}, 
we summarize the results of Krupa and Szmolyan for transcritical singularities in continous 
time~\cite{ks2001/2}. We state and explain their main result Theorem~\ref{transcritical_classic} 
in Section~\ref{contdynamics}, and we outline the main ingredients of the proof in 
Section~\ref{contblowup} thereby introducing the basic ideas of the blow-up technique. In 
Section~\ref{discretemain}, we discuss the problem in discrete time associated with~\eqref{map_intro}. 
Our main result is Theorem~\ref{transcritical_discrete}. The ingredients of the proof are developed 
in the following subsections. Section~\ref{sec:blow} introduces the blow-up method for the new 
discrete setting and, subsequently, the dynamics are analyzed in three different charts of the 
manifold that blows up the singularity, leading to the proof of Theorem~\ref{transcritical_discrete}. 
In Section~\ref{secK1}, we describe how trajectories enter a neighbourhood of the origin via 
the \emph{entrance chart} $K_1$ and leave this neighbourhood in the case of $\lambda < 1$. 
Section~\ref{secK2} builds the core of the proof: we analyze, how trajectories pass the origin 
depending on $\lambda$ in the \emph{scaling chart} $K_2$. In Section~\ref{secK3}, the \emph{exit} 
dynamics through chart $K_3$ are described for the case $\lambda >1$. Finally, in Section~\ref{secproof} 
we combine the findings of the previous sections to finish the proof of 
Theorem~\ref{transcritical_discrete}. We conclude with a short summary of our results and an outlook 
on future work in Section~\ref{summaryoutlook}.\medskip

\textbf{Acknowledgements:} CK and ME gratefully acknowledge support by the DFG via the SFB/TR109
Discretization in Geometry and Dynamics.

\section{Transcritical singularity in continuous time} 
\label{continuoustime}

We start with a brief review and notation for continuous-time fast-slow systems. 
Consider a system of singularly perturbed ordinary differential equations (ODEs) 
of the form
\begin{align} \label{slowequ}
\begin{array}{rcrcl}
\epsilon \frac{\rmd x}{\rmd \tau} &=& \epsilon \dot{x} &=& f(x,y,\epsilon)\,, \\
\frac{\rmd y}{\rmd \tau}&=&\dot{y} &=& g(x,y,\epsilon)\,, \quad \ x \in \mathbb{R}^m, 
\quad y \in \mathbb R^n, \quad 0 < \epsilon \ll 1\,,
\end{array}
\end{align}
where $f,g,$ are $C^k$-functions with $k \geq 3$. Since $\epsilon$ is a small parameter, 
the variables $x$ and $y$ are often called the \textit{fast} variable(s) and 
the \textit{slow} variable(s) respectively. The time variable 
in~\eqref{slowequ}, denoted by $\tau$, is termed the \textit{slow} time scale. The change of 
variables to the \textit{fast} time scale $t:= \tau / \epsilon$ transforms the system~\eqref{slowequ} 
into the ODEs
\begin{align} \label{fastequ}
\begin{array}{r@{\;\,=\;\,}r}
x' & f(x,y,\epsilon)\,, \\
y' & \epsilon g(x,y,\epsilon)\,.
\end{array}
\end{align}
Both equations correspond with a respective limiting problem for $\epsilon = 0$: 
the \textit{reduced problem} (or \textit{slow subsystem}) is given by
\begin{align} \label{redequ}
\begin{array}{r@{\;\,=\;\,}l}
0 & f(x,y,0)\,, \\
\dot{y} & g(x,y,0)\,, 
\end{array}
\end{align}
and the \textit{layer problem} (or \textit{fast subsystem}) is 
\begin{align} \label{layerequ}
\begin{array}{r@{\;\,=\;\,}l}
x' & f(x,y,0)\,, \\
y' & 0\,.
\end{array}
\end{align}
We can understand the reduced problem~\eqref{redequ} as a dynamical system on 
the \textit{critical manifold} 
$$S_0= \{(x,y) \in \mathbb{R}^{n+m} \,:\, f(x,y,0) = 0 \}\,.$$
Observe that the manifold $S_0$ consists of equilibria of the layer 
problem~\eqref{layerequ}. $S_0$ is called \textit{normally hyperbolic} if 
the matrix $\textnormal{D}_xf(p)\in\mathbb{R}^{m\times m}$ for all $p\in S_0$ has no
spectrum on the imaginary axis. For a normally hyperbolic $S_0$ \textit{Fenichel 
Theory} \cite{Fenichel4,Jones,ku2015,WigginsIM} implies that for $\epsilon$ 
sufficiently small, there is a locally invariant slow manifold $S_{\epsilon}$ 
such that the restriction of~\eqref{slowequ} to $S_{\epsilon}$ is a regular
perturbation of the reduced problem~\eqref{redequ}. Furthermore, it follows from 
Fenichel's perturbation results that $S_{\epsilon}$ possesses an invariant 
stable and unstable foliation, where the dynamics behave as a small perturbation 
of the layer problem~\eqref{layerequ}.

A challenging phenomenon is the breakdown of normal hyperbolicity of $S_0$ such that 
Fenichel Theory cannot be applied. Typical examples of such a breakdown are found 
at bifurcation points $p\in S_0$, where the Jacobian $\rmD_x f(p)$ has at least one 
eigenvalue with zero real part. The simplest examples are folds or points of transversal 
self-intersection of $S$ in planar systems ($m=1=n$). In the following we focus on 
the \textit{transcritical} bifurcation in planar systems.

\subsection{Dynamics near the transcritical singularity} 
\label{contdynamics}

We briefly recall the main results for transcritical fast-slow singularities in the 
continuous-time setting from~\cite{ks2001/2}. Without loss of generality, i.e., 
up to translation of coordinates, we may just assume that the transcritical point 
coincides with the origin. Consider the system of planar ODEs on the fast time scale
\begin{align} 
\label{fastequtrans}
\begin{array}{r@{\;\,=\;\,}r}
x' & f(x,y,\epsilon)\,, \\
y' & \epsilon g(x,y,\epsilon)\,,
\end{array}
\end{align}
where the vector field $f$ satisfies the following conditions at the origin:
$$ f(0,0,0) = 0, \qquad \frac{\partial }{\partial x}f (0,0,0) = 0, \qquad 
\frac{\partial} {\partial y}f (0,0,0) = 0$$
These conditions imply that the critical manifold $S:=S_0= \{ (x,y)\in\mathbb{R}^2:f(x,y,0) = 0 \}$ 
has a self-intersection at the origin. The condition
$$ \det H_f((0,0,0)) < 0\,,$$
where $H_f$ denotes the Hessian matrix of $f$ in $x,y$, implies that this 
intersection is non-degenerate. Moreover, we require 
$$\frac{\partial^2 f}{\partial^2 x} (0,0,0) \neq 0$$ 
to guarantee that $S$ is transverse to the critical fibre $\{(x,0):x \in \mathbb{R} \}$. 
Furthermore, we assume that 
$$g(0,0,0) \neq 0\,$$
to ensure transversal slow dynamics at the origin.
The transcritical bifurcation at the origin, induced by such a system, 
can be brought~\cite{ks2001/2} to the normal form
\begin{align} \label{normalform}
\begin{array}{r@{\;\,=\;\,}l}
x' & x^2 - y^2 + \lambda \epsilon + \mathcal{R}_1(x,y,\epsilon)\,, \\
y' & \epsilon(1 + \mathcal{R}_2(x,y,\epsilon))\,,
\end{array}
\end{align}
where $\lambda > 0$ is a constant and 
\begin{equation*}
\mathcal{R}_1(x,y, \epsilon) = \mathcal{O}(x^3, x^2 y, xy^2, y^3, \epsilon x, 
\epsilon y, \epsilon^2)  \,, \quad \mathcal{R}_2 (x,y, \epsilon) 
= \mathcal{O}(x,y, \epsilon) \,.
\end{equation*}
The critical manifold $S$ is the union of four branches. We denote them by 
$S_\txta^+, S_\txta^-, S_\txtr^+, S_\txtr^-$ where $\txta$ means attracting and 
$\txtr$ repelling with respect to the fast variables and $+$ and $-$ correspond 
to the sign of the $y$-variable, see also Figure~\ref{fig:1}. We denote the 
corresponding slow manifolds for small $\epsilon > 0$ 
by $S_{\txta, \epsilon}^+, S_{\txta, \epsilon}^-, S_{\txtr, \epsilon}^+, 
S_{\txtr, \epsilon}^-$. We focus on the fate of $S_{\txta, \epsilon}^-$, when 
it is continued through a neighbourhood of $(0,0)$. For that purpose, we fix $\rho > 0$ 
and let $J$ be a small open interval around $0$ in $\mathbb{R}$, potentially depending 
on $\epsilon$. Then one can define
$$ \Delta^{\textnormal{in}} := \{(- \rho, y), \, y + \rho \in J \}, \  
\Delta_\txta^{\textnormal{out}} = \{(- \rho, y), 
\, y - \rho \in J \}, \ \Delta_\txte^{\textnormal{out}} = \{( \rho, y), \, y  \in J \}\,.$$
If $\Pi_\txta$ and $\Pi_\txte$ denote the transition maps from $\Delta^{\textnormal{in}}$ to 
$\Delta_\txta^{\textnormal{out}}$ and $\Delta_\txte^{\textnormal{out}}$ respectively, 
we can formulate the main result on the transcritical singularity \cite[Theorem 2.1]{ks2001/2}.

\begin{theorem} 
\label{transcritical_classic}
Fix $\lambda \neq 1$. There exists $\epsilon_0 > 0$ such that the following 
assertions hold for $\epsilon \in [0, \epsilon_0)$.
\begin{enumerate}
\item[(T1)] If $ \lambda > 1$, then the manifold $S_{\txta, \epsilon}^-$ passes 
through $\Delta_\txte^{\textnormal{out}}$ at a point $(\rho, h(\epsilon))$ 
where $h(\epsilon) = \mathcal{O}(\sqrt{\epsilon})$. The section $\Delta^{\textnormal{in}}$ 
is mapped by $\Pi_\txte$ to an interval containing $S_{\txta, \epsilon}^{-} \cap 
\Delta_\txte^{\textnormal{out}}$ of size $\mathcal{O}( \txte^{-C/\epsilon})$, where 
$C$ is a positive constant.
\item[(T2)] If $ \lambda < 1$, then the section $\Delta^{\textnormal{in}}$ (including the 
point $\Delta^{\textnormal{in}} \cap S_{\txta, \epsilon}^-$) is mapped by $\Pi_\txta$ 
to an interval about $S_{\txta, \epsilon}^{+}$ of size 
$\mathcal{O}(\txte^{-C/\epsilon})$, where $C$ is a positive constant.
\end{enumerate}
\end{theorem}

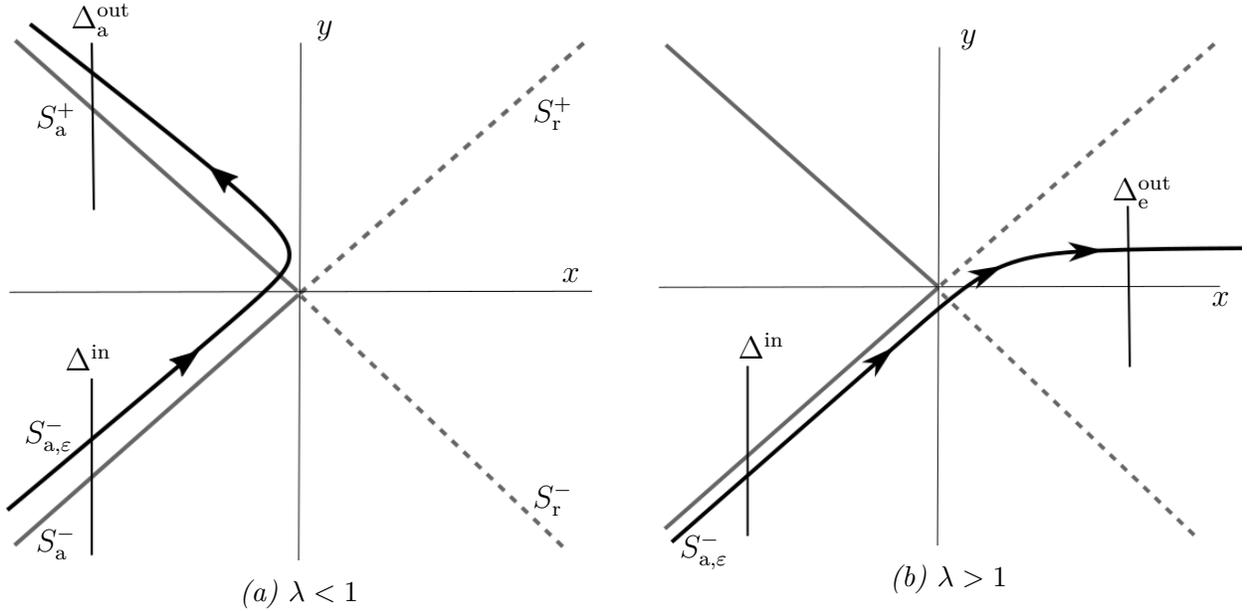
\begin{figure}[htbp]
        \centering
        \begin{subfigure}{.5\textwidth}
        \centering
  		\begin{overpic}[width=.9\textwidth]{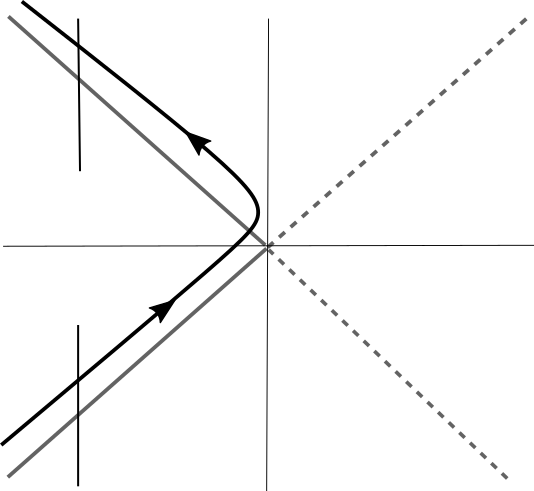}            
                \put(10,32.5){$\Delta^{\textnormal{in}}$}
                \put(11,91){$\Delta_\txta^{\textnormal{out}}$}
                \put(3,20){$S_{\txta, \epsilon}^{-}$}
								\put(5,1.5){$S_\txta^-$}
								\put(5,74){$S_\txta^+$}
								\put(90,9){$S_\txtr^-$}
								\put(90,75){$S_\txtr^+$}
								\put(95,47.5){$x$} 
								\put(53,90){$y$}
        \end{overpic}
        \caption{$\lambda < 1$}
        \label{cont_lambdasmall}
		\end{subfigure}%
        \begin{subfigure}{.5\textwidth}
        \centering
  		\begin{overpic}[width=.9\textwidth]{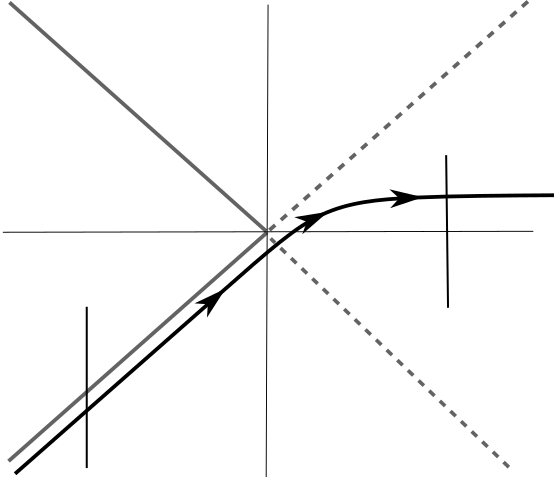}   
  				      \put(14,32){$\Delta^{\textnormal{in}}$}
                \put(78,59){$\Delta_\txte^{\textnormal{out}}$}
                \put(4,-2){$S_{\txta, \epsilon}^{-}$}
								\put(95,41){$x$} 
								\put(52,86){$y$} 
        \end{overpic}
        \caption{$\lambda > 1$}
        \label{cont_lambdalarge}
		\end{subfigure}
		\caption{Extension of the slow manifold $S_{\txta, \epsilon}^-$ (black curve) for 
		the system~\eqref{normalform} around the non-hyperbolic singularity at 
		the origin for $\lambda < 1$ and $\lambda > 1$. The four parts of the 
		critical manifold are indicated as well (grey lines), where attracting parts are
		solid lines and repelling ones dashed lines. The two sketches illustrate 
		the different cases in Theorem~\ref{transcritical_classic}.}
        \label{fig:1}
\end{figure}

In the following, we are going to sketch the proof of Theorem~\ref{transcritical_classic} 
to illustrate the setting for the continuous-time case~\cite{ks2001/2} and to be able to 
have this case available for comparison. To start, we consider 
$\epsilon$ as variable, and write the problem in three variables
\begin{align} \label{threevariables}
x' &= x^2 - y^2 + \lambda \epsilon + \mathcal{R}_1(x,y,\epsilon)\,, \nonumber \\
y' &= \epsilon (1+ \mathcal{R}_2(x,y,\epsilon))\,, \\
\epsilon' &= 0\,\nonumber.
\end{align}
The total derivative of the above vector field $X$ in $(x,y,\epsilon)$ has only 
zero eigenvalues at the origin $(x,y,\epsilon)=(0,0,0)$. In particular, the 
origin is a non-hyperbolic equilibrium. To gain hyperbolicity at the singularity one can
use the \textit{blow-up method}, which maps the equilibrium point to an entire 
manifold, on which the dynamics can be desingularized. 
The shortest proof of Theorem~\ref{transcritical_classic} uses the quasi-homogeneous 
\textit{blow-up transformation} 
$$ x = r \overline x, \quad y = r \overline y, \quad \epsilon = 
r^2 \overline \epsilon, $$
where $(\overline{x}, \overline{y}, \overline{\epsilon}, r) \in B := S^2 \times [0, r_0]$ 
for some $r_0 > 0$. The transformation can be formalized as a map $\Phi: B \to \mathbb R^3$, 
where $r_0$ is small enough such that the dynamics on $\Phi(B)$ can be described by the normal 
form approximation. The map $\Phi$ induces a vector field $\overline{X}$ on $B$ by 
$\Phi_{*}(\overline{X}) = X$, where $\Phi_{*}$ is the pushforward induced by $\Phi$. Note 
that, since $\Phi(B)$ is a full neighbourhood of the origin, it suffices to study the vector 
field $\overline{X}$. The associated dynamics on $B$ are best analysed in three charts $K_i$, 
$i = 1,2,3$, an entrance ($\overline x = -1$), a scaling ($\overline \epsilon = 1$) and an exit 
($ \overline x = 1$) chart, given by
\begin{align}
K_1 : \quad x &= - r_1, \quad y = r_1 y_1, \quad \epsilon = r_1^2 \epsilon_1, \label{K1} \\ 
K_2 : \quad x &= r_2 x_2, \quad y = r_2 y_2, \quad \epsilon = r_2^2 , \label{K2} \\
K_3 : \quad x &= r_3, \quad y = r_3 y_3, \quad \epsilon = r_3^2 \epsilon_3. \label{K3}
\end{align}
The changes of coordinates between the charts look as follows:
\noindent $k_{12}: K_1 \to K_2$ is given by
\begin{equation} \label{kappa12}
x_2 = - \epsilon_1^{-1/2}, \quad y_2 =\epsilon_1^{-1/2} y_1, \quad 
r_2 = \epsilon_1^{1/2} r_1, 
\end{equation}
\noindent $k_{21}: K_2 \to K_1$ is given by
\begin{equation} \label{kappa21}
\epsilon_1 = x_2^{-2}, \quad  y_1 = -x_2^{-1} y_2, \quad r_1 = - x_2 r_2, 
\end{equation}
\noindent $k_{32}: K_3 \to K_2$ is given by
\begin{equation} \label{kappa32}
x_2 = \epsilon_3^{-1/2}, \quad y_2 =\epsilon_3^{-1/2} y_3, \quad r_2 = \epsilon_3^{1/2} r_3, 
\end{equation}
\noindent and $k_{23}: K_2 \to K_3$ is given by
\begin{equation} \label{kappa23}
\epsilon_3 = x_2^{-2}, \quad  y_3 = x_2^{-1} y_2, \quad r_3 =  x_2 r_2.
\end{equation}

\subsection{Describing the dynamics via the blow-up method} 
\label{contblowup}
The dynamics in $K_1$ and $K_3$ can be desingularised. Indeed, the origin is
mapped to a sphere $S^2\times \{r=0\}$, and dividing the 
three vector fields by $r_i$ for $i =1,3$ and using a time change yield:
\begin{align} 
\label{K1_dynamics}
r_1' &= - r_1 F_1(r_1, y_1, \epsilon_1), \nonumber \\
y_1' &= \epsilon_1 + y_1 F_1(r_1, y_1, \epsilon_1) + \mathcal{O}(r_1), \nonumber \\
\epsilon_1' &= 2 \epsilon_1 F_1(r_1, y_1, \epsilon_1),
\end{align}
where $F_1(y_1, \epsilon_1) = 1- y_1^2 + \lambda \epsilon_1 + \mathcal{O}(r_1)$, and
\begin{align} \label{K3_dynamics}
r_3' &=  r_3 F_3(r_3, y_3, \epsilon_3), \nonumber \\
y_3' &= \epsilon_3 - y_3 F_3(r_3, y_3, \epsilon_3) + \mathcal{O}(r_3), \nonumber \\
\epsilon_3' &= -2 \epsilon_3 F_3(r_3, y_3, \epsilon_3),
\end{align}
where $F_3(y_1, \epsilon_1) = 1- y_3^2 + \lambda \epsilon_3 + \mathcal{O}(r_3)$. 
The $\mathcal{O}(r_i)$-terms are of higher order, derived from $\mathcal{R}_1$ and $\mathcal{R}_2$ in the 
original equation~\eqref{normalform}. There are six equilibria on the invariant circle 
$\{ r = \overline \epsilon = 0\}$; see also Figure~\ref{fig:2}. We denote by
\begin{equation}
p_{\txta,1}^- = (0, -1,0), \quad q_1^{\textnormal{in}} = (0,0,0), \quad p_{\txta,1}^+ = (0,1,0)
\end{equation}
the resulting equilibria in $K_1$. The points $p_{\txta,1}^+$ and $p_{\txta,1}^-$ have a 
one-dimensional stable eigenspace along the $y_1$-direction with eigenvalue $-2$ and a 
two-dimensional centre eigenspace. The point $q_1^{\textnormal{in}}$ is a saddle with eigenvalues 
$-1,1,2$. Similarly, we denote by
\begin{equation}
p_{\txtr,3}^- = (0, -1,0), \quad q_3^{\textnormal{out}} = (0,0,0), \quad p_{r,3}^+ = (0,1,0)
\end{equation}
the equilibria in $K_3$. The points $p_{\txtr,3}^+$ and $p_{\txtr,3}^-$ have a one-dimensional 
unstable eigenspace along the $y_3$-direction with eigenvalue $2$ and a two-dimensional centre 
eigenspace. The 
point $q_1^{\textnormal{out}}$ is a saddle with eigenvalues $1,-1,-2$. Note that we have two 
hyperbolic equilibria  and four partially hyperbolic equilibria on $\{ r = \overline \epsilon = 0\}$ 
as opposed to complete lack of hyperbolicity in the original problem, see Figure~\ref{fig:2}. 

The centre manifolds $M_{\txta,1}^{\pm}$ of $p_{\txta,1}^{\pm}$ in chart $K_1$ and $M_{\txtr,3}^{\pm}$ 
of $p_{r,3}^{\pm}$ in chart $K_3$ are the unique extensions of slow manifolds $S_{\txta,1}^{\pm}$ 
and $S_{\txtr,3}^{\pm}$, which correspond to $S_{\txta,\epsilon}^{\pm}$ 
and $S_{\txtr,\epsilon}^{\pm}$ in the original coordinates. Furthermore, $ M_{\txta,1}^{\pm}$ 
and $ M_{\txtr,3}^{\pm}$ correspond to locally invariant manifolds $\overline M_{\txta}^{\pm}$ 
and $\overline M_{\txtr}^{\pm}$ of the blown-up vector field $\overline{X}$ at the equilibria 
$\overline p_{\txta}^{\pm}$ and $\overline p_{\txtr}^{\pm}$. The idea is to track $M_{\txta}^{-}$ 
as it moves across the sphere $S^2$ guided by the dynamics in chart $K_2$.
For that purpose, it is helpful to introduce the centre manifolds $ N_{\txta,1}^{\pm} = M_{\txta,1}^{\pm} 
\cap \{r_1 = 0\}$ and $ N_{\txtr,3}^{\pm} = M_{\txtr,3}^{\pm} \cap \{r_3 = 0\}$ and their global 
counterparts $N_{\txta}^{\pm}$ and $N_{\txtr}^{\pm}$. In that way, the dynamics can be analyzed first for 
$r=0$, i.e.~at the origin now blown up to a sphere where partial hyperbolicity is gained, and then for 
small $r > 0$ in order to connect the dynamics around the origin.

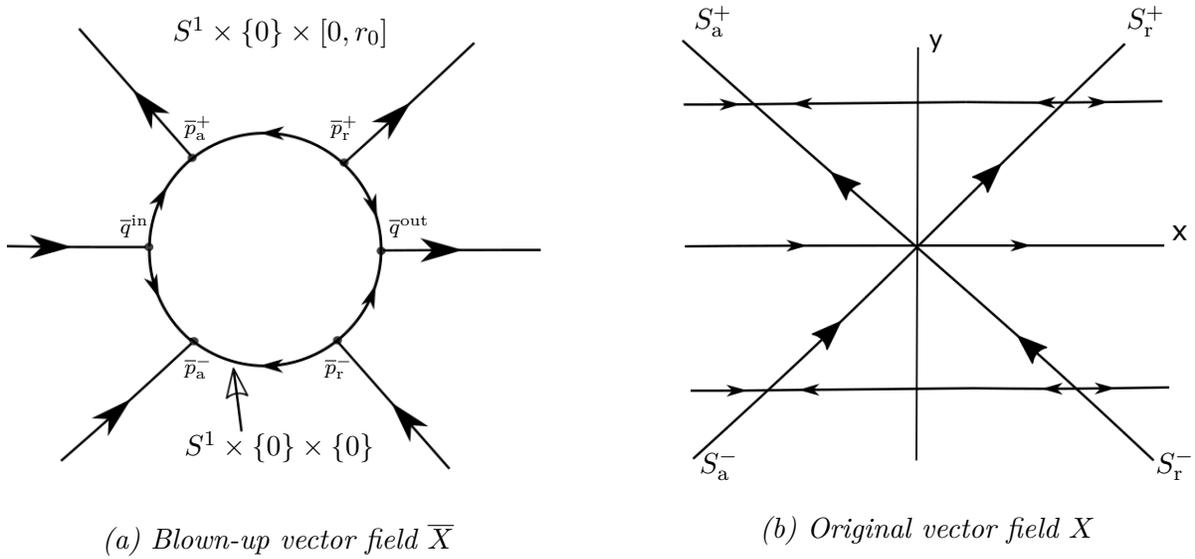
\begin{figure}[htbp]
        \centering
        \begin{subfigure}{.5\textwidth}
        \centering
  		\begin{overpic}[width=.9\textwidth]{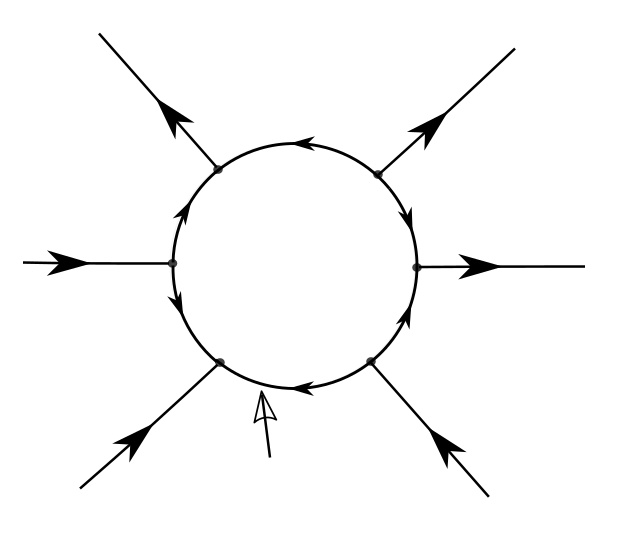}   
  		\put(32,80){\small $S^1 \times \{0\} \times [0,r_0]$}         
                \put(34,64){\scriptsize $\overline p_{\txta}^{+}$}
                \put(34,23){\scriptsize $\overline p_{\txta}^{-}$}
                \put(58,23){\scriptsize $\overline p_{\txtr}^{-}$}
                \put(59,64){\scriptsize $\overline p_{\txtr}^{+}$}
                \put(23,47){\scriptsize $\overline q^{\textnormal{in}}$}
                \put(69,47){\scriptsize $\overline q^{\textnormal{out}}$}
                \put(34,9){\small$S^1 \times \{0\} \times \{0\}$}
        \end{overpic}
        \caption{Blown-up vector field $\overline{X}$}
        \label{blownup_vf}
		\end{subfigure}%
        \begin{subfigure}{.5\textwidth}
        \centering
  		\begin{overpic}[width=.9\textwidth]{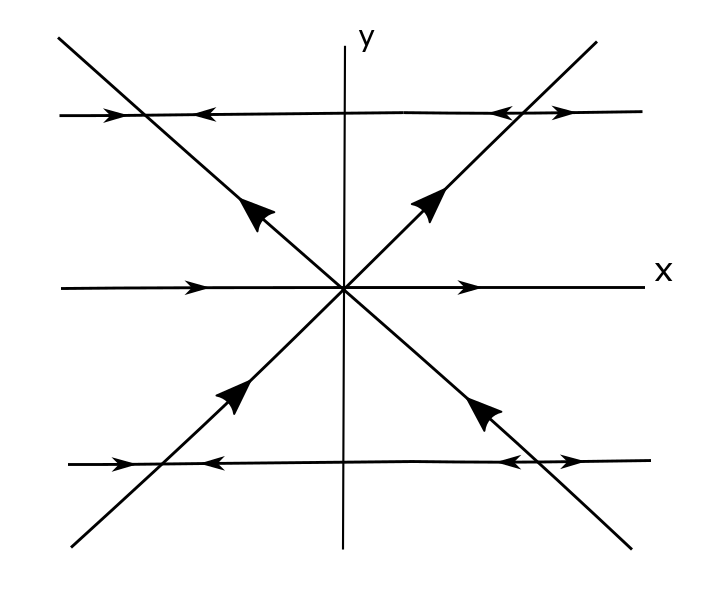}   
  				\put(84,80){\small $S_{\txtr}^{+}$}
                \put(89,4){\small $S_{\txtr}^{-}$} 
                \put(10,80){\small $S_{\txta}^{+}$}
                \put(11,4){\small $S_{\txta}^{-}$}         
        \end{overpic}
        \caption{Original vector field $X$}
        \label{original_vf}
		\end{subfigure}
		\caption{Blow-up at $(0,0,0)$ of equation~\eqref{threevariables}. Figure 2 (b) shows a sketch of the original vector field $X$ for $\epsilon =0$ and the four different branches of the critical manifold $S_{\txta}^{-},S_{\txtr}^{-}, S_{\txta}^{+}, S_{\txtr}^{+}$. The arrows on the critical manifold indicate the direction of the reduced flow whereas the arrows on the $y$-fibres illustrate the layer flow. Figure 2 (a) sketches the blown-up vector field $\overline{X}$ for $\epsilon=0$, i.e. on the submanifold $S^1 \times \{0\} \times [0,r_0]$. The figure illustrates the partially hyperbolic dynamics around the six equilibrium points on the invariant circle $S^1 \times \{0\} \times \{0\}$ which corresponds with the singularity at the origin in Figure 2 (b).}
        \label{fig:2}
\end{figure}

The crucial dynamics happen in the scaling chart $K_2$. The desingularized equations have the form
\begin{align} \label{K2_dyanmics}
x_2' &= x_2^2 - y_2^2 + \lambda + \mathcal{O}(r_2)\, \nonumber \\
y_2' &= 1 + \mathcal{O}(r_2)\,\,, \nonumber \\
r_2' &= 0\,.
\end{align}
Firstly, taking $r_2 = 0$, one can develop the typical behaviour of trajectories around the 
canard case $\lambda =1$. In this case, there is the solution $x_2 = y_2 = t$ denoted by 
$\gamma_2(t)$. We observe that
$$ \lim_{t \to \infty} \kappa_{23}(\gamma_2(t)) = p_{\txtr,3}^+\,, \quad \lim_{t \to -\infty} 
\kappa_{21}(\gamma_2(t)) = p_{\txta,1}^-\,. $$
Hence, for $\lambda =1 $, there is a connection between the equilibrium points $\overline p_\txta^-$ 
and $\overline p_\txtr^+$ by a trajectory $ \overline \gamma$ which equals $\gamma_2$ in $K_2$. 
Krupa and Szmolyan \cite[Lemma~2.2]{ks2001/2} argue with the rotational property of \eqref{K2_dyanmics} 
to observe that a connection from $\overline p_\txta^-$ to $\overline p_\txtr^+$ only exists for 
$\lambda =1$. They conclude \cite[Proposition 2.1]{ks2001/2} that for $\lambda < 1$ there is a unique 
trajectory connecting $\overline p_\txta^-$ to $\overline p_\txta^+$ corresponding to $\overline N_\txta^-$, 
and for $\lambda > 1$ there is a unique trajectory connecting $\overline p_\txta^-$ to $\overline 
q^{\textnormal{out}}$ corresponding to $\overline N_\txta^-$; see Figure~\ref{fig:2}. 

Moreover, one can argue that for $\lambda \neq 1$ the $\mathcal{O}(r_2)$ terms do not change 
the qualitative picture: if $\lambda < 1$, it follows from the previous considerations and 
perturbation theory that $\overline M_\txta^-$ reaches the vicinity of $\overline M_\txta^+$ 
and that both centre manifolds are exponentially attracting. Analyzing the dynamics 
of~\eqref{K1_dynamics} on $\overline M_\txta^{\pm}$ shows that the passage times near the 
equilibria $\overline p_\txta^{\pm}$ are of order $\mathcal{O}(1/\epsilon)$. Hence, one can 
conclude that the relevant contraction rate is of order $\mathcal{O}(\txte^{-C/\epsilon})$ for 
some $C > 0$, as stated in Theorem~\ref{transcritical_classic}.

If $\lambda > 1$, it follows from the previous considerations and perturbation theory that 
$\overline M_\txta^-$ passes near $\overline q^{\textnormal{out}}$. Contraction towards 
$\overline M_\txta^-$ works as for $\lambda < 1$. The fact that $h(\epsilon) = 
\mathcal{O}(\sqrt{\epsilon})$ has to be worked out from the asymptotics of $\overline N_\txta^-$ 
in chart $K_2$ and an analysis of the linearization of $\overline{X}$ at 
$\overline q^{\textnormal{out}}$ in chart $K_3$. The latter is delicate due to the resonance of 
the eigenvalues $-1,1$ and $2$. Krupa and Szmolyan discuss this problem in detail for the fold 
singularity \cite[Section 2.6]{ks2011} but not for the transcritical problem and claim that the 
statement follows analogously. We are going to give a detailed argument for the discrete-time 
problem below.

\section{Transcritical singularity in discrete time} 
\label{discretemain}

We can now turn to the main part, i.e., we want to analyze the discrete-time problem obtained via 
an explicit Euler method. For that purpose, we first 
set the higher order terms in~\eqref{threevariables}, represented by $\mathcal{R}_1$ and $\mathcal{R}_2$, to zero. 
We introduce the step size $h > 0$ of the Euler method as an additional variable and obtain a 
map $P: \mathbb{R}^4 \to \mathbb{R}^4$, whose iterations $P^n(z_0)$, for $n \in \mathbb{N}$ and 
$z_0 \in \mathbb{R}^4$,  we are going to analyze close to the origin with $h, \epsilon > 0$:
\begin{equation} 
\label{map_transcritical}
P: \begin{pmatrix}
x \\ y \\ \epsilon \\ h
\end{pmatrix}
\mapsto \begin{pmatrix}
\tilde{x}  \\ \tilde y \\ \tilde \epsilon \\ \tilde h 
\end{pmatrix}
= \begin{pmatrix}
x + h(x^2 - y^2 + \lambda \epsilon) \\
y + \epsilon h \\
\epsilon \\
h
\end{pmatrix}.
\end{equation}
Also in this case there is a normally hyperbolic critical manifold 
$$S_0 = \{(x,y,\epsilon,h)\in\mathbb{R}^4:x^2 = y^2\} \setminus \{0\},$$ which 
splits into the attracting branches $S_\txta^-, S_\txta^+$ and repelling branches 
$S_\txtr^-, S_\txtr^+$. It follows from~\cite[Theorem 4.1]{HPS77} that for $\epsilon, 
h > 0$ small enough there are corresponding forward invariant slow manifolds 
$S_{\txta, \epsilon,h}^-, S_{\txta, \epsilon,h}^+$ and $S_{\txtr, \epsilon,h}^-, 
S_{\txtr, \epsilon,h}^+$. 
Note that $\rmD P(0)$ only has quadruple eigenvalue $1$, which means a complete loss of 
hyperbolicity at the origin, as in the ODE case.

Similarly to the problem in continous time, we fix some $\rho > 0$ and let 
$J$ be a small open interval around $0$ in $\mathbb{R}$. We define
$$ \Delta^{\textnormal{in}} = \{(- \rho, y), \, y + \rho \in J \}, \  
\Delta_\txta^{\textnormal{out}} = \{(- \rho, y), \, y - \rho \in J \}, 
\ \Delta_\txte^{\textnormal{out}} = \{( \rho, y), \, y  \in J \},$$
where $\varepsilon$ and $h$ are fixed as prescribed by the map $P$; see also 
Figure~\ref{fig:discrete1}. In contrast with flows, the intervals 
$\Delta_\txta^{\textnormal{out}}$, $\Delta_\txte^{\textnormal{out}}$ are not 
necessarily hit by $P^n(-\rho,y)$ for fixed $y \in \Delta^{\textnormal{in}}$ and 
some $n > 0$. Notice that we used an abbreviated notation $P^n(x,y,\varepsilon,h)=P^n(x,y)$ 
for the map $P$ and also for points in $\Delta^{\textnormal{in}}$ just denoting them 
by their $y$-coordinate. We define the transition maps from $\Delta^{\textnormal{in}}$ 
to the vicinity of $\Delta_\txta^{\textnormal{out}}$ and 
$\Delta_\txte^{\textnormal{out}}$ by
\begin{align} 
\label{maps}
\Pi_\txta(y) = P^{n^*(y)}(-\rho,y)\,, \text{ where } n^*(y)  &= 
\argmin_{n \in \mathbb{N}} \dist( P^{n}(-\rho,y), \Delta_\txta^{\textnormal{out}})\,, 
\ y \in \Delta^{\textnormal{in}}\,, \\
\Pi_\txte(y) = P^{m^*(y)}(-\rho,y)\,, \text{ where } m^*(y) &= \argmin_{n \in \mathbb{N}} 
\dist( P^{n}(-\rho,y), \Delta_\txte^{\textnormal{out}})\,, \ y \in \Delta^{\textnormal{in}}.
\end{align}
We can formulate the main result on the transcritical singularity in discrete time (see Figure~\ref{fig:discrete1} for an illustration):

\begin{theorem} 
\label{transcritical_discrete}
Fix $\lambda \neq 1$ and $\rho > 0$. There exists $\epsilon_0 > 0$ such that the 
following assertions hold for all $\epsilon \in [0, \epsilon_0]$, $h > 0$ such that 
$0 < h \rho^3 < \epsilon$ and any $0 < c < \rho h$.
\begin{enumerate}
\item[(T1)] If $ \lambda < 1$, then the section $\Delta^{\textnormal{in}}$ (including the 
point $\Delta^{\textnormal{in}} \cap S_{\txta, \epsilon, h}^-$) is mapped by $\Pi_\txta$ 
to a set about $S_{\txta, \epsilon, h}^{+}$ of $y$-width $\mathcal{O}\left( 
(1- c)^{\frac{C \rho }{h \epsilon}} \right)$, where $C$ is a positive constant, such that 
every point in $\Pi_\txta \left( \Delta^{\textnormal{in}} \right)$ is 
$\mathcal{O}(h \epsilon)$-close to $\Delta_\txta^{\textnormal{out}}$.
\item[(T2)] If $ \lambda > 1$, then the manifold $S_{\txta, \epsilon,h}^-$ passes through 
$\Delta_\txte^{\textnormal{out}}$ at a point $(\rho, k(\epsilon))$ where $k(\epsilon) = 
\mathcal{O}(\epsilon^{1/3})$. The section $\Delta^{\textnormal{in}}$ is mapped by $\Pi_\txte$ 
to a set about $S_{\txta, \epsilon,h}^{-}$ of $y$-width $\mathcal{O}\left( (1- c)^{\frac{C
 \rho}{ h \epsilon}}\right)$, where $C$ is a positive constant, such that every point in 
$\Pi_\txte \left( \Delta^{\textnormal{in}} \right)$ is $\mathcal{O}(h (\epsilon + 
\rho^2))$-close to $\Delta_\txte^{\textnormal{out}}$.
\end{enumerate}
\end{theorem}

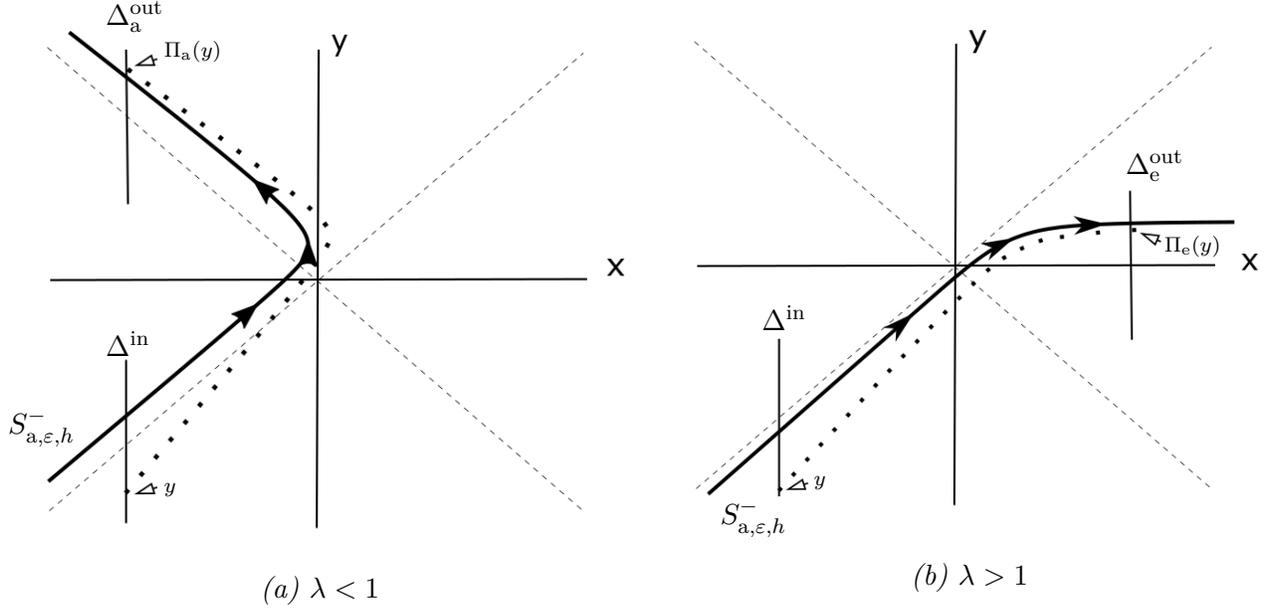
\begin{figure}[htbp]
        \centering
        \begin{subfigure}{.5\textwidth}
        \centering
  		\begin{overpic}[width=1.0\textwidth]{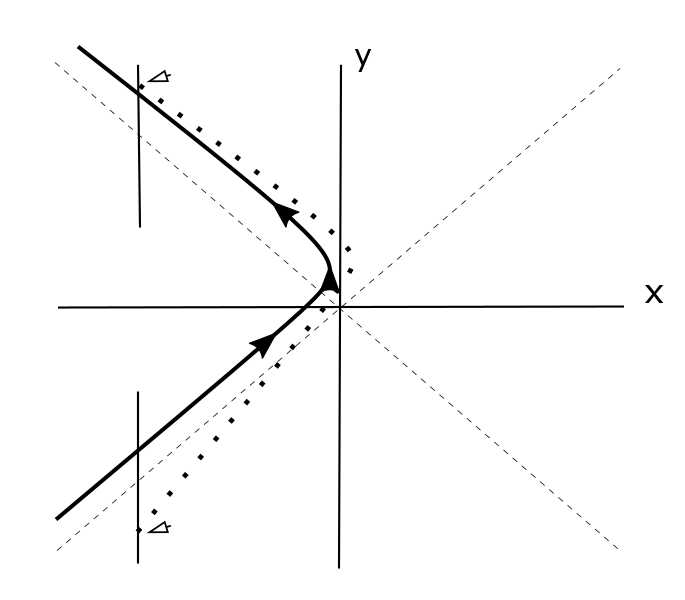}            
                \put(17,31){\small $\Delta^{\textnormal{in}}$}
                \put(17,82){\small $\Delta_\txta^{\textnormal{out}}$}
                \put(2,19){\small $S_{\txta, \epsilon,h}^{-}$}
                \put(26,10){\scriptsize $y$}
                \put(26,77){\scriptsize $\Pi_\txta(y)$}
        \end{overpic}
        \caption{$\lambda < 1$}
        \label{discrete_lambdasmall}
		\end{subfigure}%
        \begin{subfigure}{.5\textwidth}
        \centering
  		\begin{overpic}[width=1.0\textwidth]{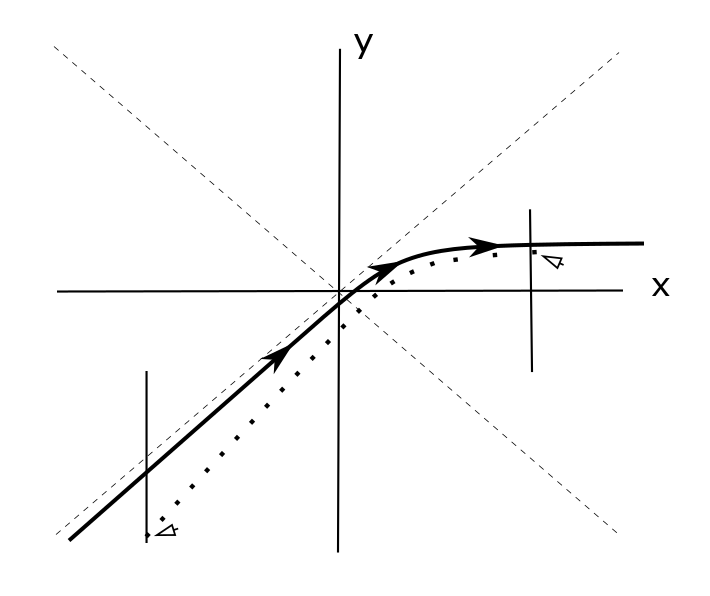}   
  				\put(18,33){\small $\Delta^{\textnormal{in}}$}
                \put(74,57){\small $\Delta_\txte^{\textnormal{out}}$}
                \put(10,3){ \small $S_{\txta, \epsilon,h}^{-}$}         
                \put(26,9){\scriptsize $y$}
                \put(80,45){\scriptsize $\Pi_\txte(y)$}
        \end{overpic}
        \caption{$\lambda > 1$}
        \label{discrete_lambdalarge}
		\end{subfigure}
		\caption{The sketches illustrate Theorem~\ref{transcritical_discrete}. For $\lambda < 1$ (Figure 3 (a))
		and $\lambda > 1$ (Figure 3 (b)), the figures show the extension of the slow manifold $S_{\txta, \epsilon,h}^-$ (black curve)
		for system~\eqref{map_transcritical} around the non-hyperbolic singularity at the origin 
		and the trajectory of a point $y \in \Delta^{\textnormal{in}}$ (dotted curve) to $\Pi_\txta(y)$ near $\Delta_\txta^{\textnormal{out}}$ and to
		$\Pi_\txte(y)$ near  $\Delta_\txte^{\textnormal{out}}$ respectively . The dashed lines show the branches of the critical manifold $S_0$.}
        \label{fig:discrete1}
\end{figure}

Note carefully that Theorem~\ref{transcritical_discrete} includes a precise 
requirement between three parameters, i.e., $0 < h \rho^3 < \epsilon$, which
means that the choice of step size for the Euler scheme is crucial. Since we only work in the normal form, the parameter $\rho$ does not have to be small and can, for example, be chosen equal to $1$ such that the requirement reads $0 < h < \epsilon$.  Our aim is to 
prove Theorem~\ref{transcritical_discrete} using the blow-up method for the problem 
in four variables and to track individual trajectories inside the slow manifolds.

\subsection{Blow-up transformation} 
\label{sec:blow}

We conduct the quasihomogeneous blow-up transformation
$$ x = \bar r \bar x, \quad y = \bar r \bar y, \quad \epsilon = 
\bar r^2 \bar \epsilon, \quad h = \bar h/\bar r\,, $$
where $(\bar x, \bar y, \bar \epsilon, \bar h, \bar r) \in B := S^2 \times 
(0, h_0] \times (0, \rho]$ for some $h_0, \rho > 0$. The change of variables in 
$h$ is chosen such that the map is desingularized in the relevant charts. We exclude 
$0$ from the domain of $\bar h$ since at $\bar h = 0$ every point is a neutral fixed 
point. Due to the transformation $h = \bar h/\bar r$ we have to exclude $0$ from the 
domain of $\bar r$ as well.

The whole transformation can be formalised as a map $\Phi: B \to \mathbb R^4$. 
The map $\Phi$ induces a map $\overline{P}$ on $B$ by $\Phi \circ \overline{P} 
\circ \Phi^{-1} = P$. Analogously to the continuous time case, we are using the 
charts $K_i$, $i=1,2,3$, to describe the dynamics. The chart $K_1$ focuses on 
the entry of trajectories for any value of lambda and the exit of trajectories 
for $\lambda < 1$, and is given by
\begin{equation} \label{K1d}
x = - r_1, \quad y = r_1 y_1, \quad \epsilon = r_1^2 \epsilon_1, \quad h = h_1/r_1\,.
\end{equation}
In the scaling chart $K_2$ the dynamics arbitrarily close to the origin are analyzed. 
It is given via the mapping
\begin{equation} \label{K2d}
x = r_2 x_2, \quad y = r_2 y_2, \quad \epsilon = r_2^2 , \quad h = h_2/r_2\,.
\end{equation}
The exit chart $K_3$ plays a role for the dynamics emerging from a neighbourhood of 
the origin for $\lambda > 1$ and is given by
\begin{equation} \label{K3d}
x = r_3, \quad y = r_3 y_3, \quad \epsilon = r_3^2 \epsilon_3, \quad h = h_3/r_3\,.
\end{equation}
There are four relevant changes of coordinates between the charts.
\noindent The map $k_{12}: K_1 \to K_2$ is given by
\begin{equation} \label{kappa12d}
x_2 = - \epsilon_1^{-1/2}, \quad y_2 =\epsilon_1^{-1/2} y_1, \quad 
r_2 = \epsilon_1^{1/2} r_1, \quad h_2 = \epsilon_1^{1/2} h_1\,,
\end{equation}
\noindent $k_{21}: K_2 \to K_1$ is given by
\begin{equation} \label{kappa21d}
\epsilon_1 = x_2^{-2}, \quad  y_1 = -x_2^{-1} y_2, \quad 
r_1 = - x_2 r_2, \quad h_1 = - x_2 h_2\,,
\end{equation}
\noindent $k_{32}: K_3 \to K_2$ is given by
\begin{equation} \label{kappa32d}
x_2 = \epsilon_3^{-1/2}, \quad y_2 =\epsilon_3^{-1/2} y_3, \quad 
r_2 = \epsilon_3^{1/2} r_3, \quad h_2 = \epsilon_3^{1/2} h_3\,,
\end{equation}
\noindent and $k_{23}: K_2 \to K_3$ is given by
\begin{equation} \label{kappa23d}
\epsilon_3 = x_2^{-2}, \quad  y_3 = x_2^{-1} y_2, \quad 
r_3 =  x_2 r_2, \quad h_3 =  x_2 h_2.
\end{equation}
\subsection{Dynamics in the chart $K_1$} \label{secK1}
We choose $\delta > 0$ small such that $ \left| \lambda \delta \right| \leq 1$, to 
be determined later in more detail which will also determine $\epsilon_0 = \rho^2 
\delta$. Furthermore, we assume $\nu := \rho h < \delta$ for fixed $h \in (0, h_0]$. 
We are interested in trajectories entering $B$ at $\bar r = \rho$ which is best 
analyzed in the entering chart $K_1$. At $\bar r = \rho$ we have $h_1 = \nu$.
We investigate the dynamics within the domain
$$D_1 := \{(r_1, y_1, \epsilon_1, h_1) \in \mathbb{R}^4 : r_1 \in 
[ 0, \rho], \epsilon_1 \in [0, 2 \delta], h_1 \in [ 0, \nu] \}\,.$$
Note that we have to bound $h_1$ from below since for $h_1 = 0$ everything is fixed 
and it is helpful to choose a uniform bound to get estimates on the contraction rates. 
A suitable choice is $h_1 \geq \nu/2$. The proportionality $h= h_1/r_1$ implies that 
$r_1 \geq \rho/2$. Furthermore, we want to see what happens for $\epsilon_1 = \delta$. 
Due to the invariant relation $\epsilon_1 r_1 h_1 = \epsilon h$, this implies taking $\epsilon_1 
\geq \delta/4$. These considerations lead to introducing the subdomain $\hat D_1 \subset 
D_1$ which is given as
$$ \hat D_1 := \{(r_1, y_1, \epsilon_1, h_1) \in \mathbb{R}^4 : r_1 \in [ \rho/2, \rho], 
\epsilon_1 \in [\delta/4, \delta], h_1 \in [ \nu/2, \nu] \}\,.$$
We will later restrict $y_1$ to obtain a small neighbourhood of $\Delta^{\textnormal{in}}$ 
as entering domain. To derive the blown-up map we calculate
\begin{align*}
\tilde r_1 &= - \tilde x = - x - h(x^2 - y^2 + \lambda \epsilon) \\
&= r_1 - \frac{h_1}{r_1}(r_1^2 - r_1^2 y_1^2 + \lambda r_1^2 \epsilon_1) \\
&= r_1(1 - h_1(1-y_1^2 + \lambda \epsilon_1))\,.
\end{align*}
Similarly, we can derive the maps for the other variables in chart $K_1$ leading to 
the following dynamics, desingularised by choosing $h = h_1/r_1$:
\begin{align} \label{K1dynamics}
\tilde{r}_1 &= r_1(1 - h_1F_1(y_1, \epsilon_1)), \nonumber \\
\tilde y_1 &= (y_1 + \epsilon_1 h_1)(1 - h_1F_1(y_1, \epsilon_1))^{-1},\nonumber \\
\tilde \epsilon_1 &=  \epsilon_1 (1 - h_1F_1(y_1, \epsilon_1))^{-2},\nonumber \\
\tilde{h}_1 &= h_1(1 - h_1F_1(y_1, \epsilon_1)),
\end{align}
where $F_1(y_1, \epsilon_1) = 1 - y_1^2 + \lambda \epsilon_1$. Now we have to analyze
the dynamics of~\eqref{K1dynamics} in detail. For any $h_1 \in [ 0, \nu]$ 
system~\eqref{K1dynamics} has the fixed points
$$ v_{\txta,1}^-(h_1) = (0,-1,0,h_1), \quad  v_{\txta,1}^+(h_1) = (0,1,0,h_1).$$
The points $v_{\txta,1}^-$ and $v_{\txta,1}^+$ have a three-dimensional centre eigenspace 
and a one-dimensional eigenspace spanned by $(0,1,0,0)^\top$ with the eigenvalue 
$\lambda_1 = 1-2h_1$, which is stable as long as $h_1 < 1$. Note that the set
$$ \{ w^{\textnormal{in}}(h_1) := (0,0,0,h_1) \, ; \, h_1 \in [ 0, \nu ] \}$$
is an invariant set for system~\eqref{K1dynamics} within $D_1$. The points 
$w^{\textnormal{in}}(h_1)$ have two stable and two unstable eigenvalues
$$\lambda_1 = 1 - 2h_1, \quad \lambda_2 = 1 - h_1 , \quad \lambda_3 =(1 - h_1)^{-1}, 
\quad \lambda_4 =(1 - h_1)^{-2}\,,$$ 
such that again the stability depends on $h_1$ and is analogous to the time-continuous 
case, if $h_1 < 1$. 
The eigenvalues $\lambda_1, \lambda_2$ correspond with the $h_1$- 
and $r_1$-directions and $\lambda_3, \lambda_4$ with the $y_1$-and $\epsilon_1$-directions.
Moreover, we remark that we can re-interpret the stability conditions to obtain the 
same behaviour as in the continuous time case, such as 
\begin{equation}
1>|\lambda_1| = |1 - 2h_1|,
\end{equation}
precisely as the stability criteria of the Euler method derived from the Dahlquist test equation \cite{Dahlquist}
within each eigenspace of the continuous-time blow-up problem in chart $K_1$.

We observe that the two-dimensional planes
$$ S_{\txta,1}^{\pm} = \{(r_1, y_1, \epsilon_1, h_1) \in D_1 \,:\, y_1 = \pm 1, \ 
\epsilon_1 = 0 \}$$
are invariant manifolds of $D_1$ only consisting of fixed points, attracting in the 
$y_1$-direction and neutral in the other directions. One can extend these manifolds 
$S_{\txta,1}^{\pm}$ to center-stable invariant manifolds $M_{\txta,1}^{\pm}$ (see Figure~\ref{fig:K1}), which are given in $D_1$ by graphs $y_1 = l_{\pm}(\epsilon_1, h_1)$ for mappings $l_\pm$. We can 
derive $l_{\pm}$ from the discrete invariance equation
\begin{equation} 
\label{Invariance_equ}
l_{\pm} (\tilde \epsilon_1, \tilde h_1)= \frac{l_{\pm} (\epsilon_1, h_1) + 
\epsilon_1 h_1}{1 - h_1F_1( l_{\pm} (\epsilon_1, h_1), \epsilon_1)}\,.
\end{equation} 
Solving this equation allows us to make the following statement.

\begin{proposition} 
\label{Invariance_Prop}
Equation~\eqref{Invariance_equ} has the solutions
\begin{align}
l_{-} (\epsilon_1, h_1) &= -1 + \frac{1 - \lambda}{2} \epsilon_1 + 
\mathcal{O}(\epsilon_1^2 h_1)\,, \label{lminus}\\
l_{+} (\epsilon_1, h_1) &= 1 + \frac{1 + \lambda}{2} \epsilon_1 + 
\mathcal{O}(\epsilon_1^2 h_1)\,. \label{lplus}
\end{align}
which characterize $M_{\txta,1}^{-}$ and $M_{\txta,1}^{+}$ respectively. Furthermore, 
$\epsilon_1$ is increasing on $M_{\txta,1}^{-}$ and decreasing on $M_{\txta,1}^{+}$, 
whereas $h_1, r_1$ are decreasing on $M_{\txta,1}^{-}$ and increasing on $M_{\txta,1}^{+}$.
\end{proposition}

\begin{proof}
It is easy to derive that for $l_{-} (\epsilon_1, h_1)$ given by~\eqref{lminus}, we have
\begin{equation} \label{F_manifold}
 F_1(l_{-} (\epsilon_1, h_1), \epsilon_1) = \epsilon_1 + \mathcal{O}(\epsilon_1^2).
\end{equation}
Hence, we observe that 
$$ \tilde \epsilon_1 =  \epsilon_1 (1 - h_1F_1(l_{-} (\epsilon_1, h_1), \epsilon_1))^{-2} 
= \epsilon_1 + \mathcal{O}(\epsilon_1^2 h_1)$$
and
$$ \tilde h_1 = h_1 + \mathcal{O}(h_1^2 \epsilon_1)\,.$$
Therefore, we deduce that
\begin{align*}
\frac{l_{-} (\epsilon_1, h_1) + \epsilon_1 h_1}{1 - h_1F( l_{-} 
(\epsilon_1, h_1), \epsilon_1)} &= (l_{-} (\epsilon_1, h_1) + \epsilon_1 
h_1)(1 + h_1 \epsilon_1 + \mathcal{O}(\epsilon_1^2 h_1))\\
&= l_{-} (\epsilon_1, h_1) - h_1 \epsilon_1 + \frac{1 - \lambda}{2} 
\epsilon_1^2 h_1 + \epsilon_1 h_1 + \mathcal{O}(\epsilon_1^2 h_1)\\
 &= l_{-} (\epsilon_1, h_1) + \mathcal{O}(\epsilon_1^2 h_1) = -1 + 
\frac{1 - \lambda}{2} \epsilon_1 + \mathcal{O}(\epsilon_1^2 h_1) = 
l_{-} (\tilde \epsilon_1, \tilde h_1),
\end{align*}
which shows the claim for $l_{-} (\epsilon_1, h_1)$. Since we can assume 
that $h_1 \epsilon_1 < 1$, the dynamics on $M_{\txta,1}^{-}$ follow as stated.
Similarly we can derive that for $l_{+} (\epsilon_1, h_1)$ given by~\eqref{lplus}, 
we have $$ F_1(l_{+} (\epsilon_1, h_1)) = -\epsilon_1 + \mathcal{O}(\epsilon_1^2).$$
The statements then follow analogously to before.
\end{proof}

For all trajectories, as explained above, we have to consider the entry region
\begin{equation*}
\Sigma_{1,-}^{\textnormal{in}} :=  \{(r_1, y_1, \epsilon_1, h_1) 
\in D_1 \,:\, r_1 = \rho, \ h_1 = \nu, \epsilon_1 = \delta/4 \}\,.
\end{equation*} 
Before exiting $\hat D_1$ for the first time, the dynamics must reach the set
\begin{equation*}
\Sigma_{1,-}^{\textnormal{out}} =  \left\{(r_1, y_1, \epsilon_1, h_1) 
\in \mathbb{R}^4 \,:\, \frac{\rho}{2}\leq r_1 \leq \frac{\rho}{2}(1+ \nu), 
\ \frac{\nu}{2}\leq h_1 \leq \frac{\nu}{2}(1+\nu), \ \delta(1 -2 \nu) \leq 
\epsilon_1 \leq \delta\right\}\,,
\end{equation*} 
since $F_1(y_1, \epsilon_1) \leq2$. Next, we want to find a set $R \subset \Sigma_{1,-}^{\textnormal{in}} $ such that $M_{\txta,1}^- \cap \Sigma_{1,-}^{\textnormal{in}} 
\subset R$ and there is a well-defined map $\Pi_{1,-}: R
\to \Sigma_{1,-}^{\textnormal{out}}$ that maps points in $R$ 
along a trajectory of~\eqref{K1dynamics} to a first entry point in 
$\Sigma_{1,-}^{\textnormal{out}}$.
By Proposition~\ref{Invariance_Prop}, this is feasible for $R$ small enough such that trajectories through $R$ stay sufficiently close to $M_{\txta,1}^-$ in the first 
part of the passage in $K_1$.
 
We choose $R$ to be the interval 
\begin{equation} \label{R1}
R_1 := \left\{(r_1, y_1, \epsilon_1, h_1) \in \Sigma_{1,-}^{\textnormal{in}} 
\,;\, -1- \beta_1(\lambda) \leq y_1 \leq -1 + \hat \beta_1  \right\}
\end{equation}
with, for example, 
\begin{equation} \label{betas}
\hat \beta_1 := \left| \lambda - 1\right|  \delta\,, \quad \beta_1(\lambda)
:= \begin{cases} \frac{\lambda}{16} \delta &\text{ if } 0 < \lambda < 1 \\
\frac{2 \lambda -1}{16} \delta &\text{ otherwise.}
\end{cases}  
\end{equation}
Note that with these choices we have $M_{\txta,1}^- \cap \Sigma_{1,-}^{\textnormal{in}} 
\subset R_1$ for $\nu, \delta$ sufficiently small. Furthermore, these choices guarantee 
that the trajectories stay close to $M_{\txta,1}^-$ such that $F_1(y_1, \epsilon_1)$ 
is positive, and, hence, we can formulate the following Proposition (see Figure~\ref{fig:K1}).

\begin{proposition} 
\label{Justification}
Trajectories in $\hat{D}_1$ starting in $R_1$ are increasing in $\epsilon_1$ and 
decreasing in $h_1, r_1$. Hence, the transition map $\Pi_{1,-}: R_1
\to \Sigma_{1,-}^{\textnormal{out}}$ is well-defined.
\end{proposition}

\begin{proof}
It is enough to show that in this case $F_1(y_1, \epsilon_1)$ is positive. 
If $ \lambda \geq 1$ or $\lambda \leq 0$, we observe that $ \beta_1(\lambda) 
= \frac{2 \lambda -1}{16} \delta$ implies $F_1 \geq \frac{\delta - \mathcal{O}(\delta^2)}{8}$. 
If $ 0 < \lambda < 1$, we have
$$ F_1(y_1, \epsilon_1) \geq   1 - \left(-1 - \frac{\lambda}{16} \delta\right)^2 + \lambda 
\epsilon_1  \geq  \frac{\lambda}{8} \delta  - \left( \frac{\lambda}{16} \right)^2 \delta^2  \,.$$
Together with the considerations above, we can conclude the claim.
\end{proof}

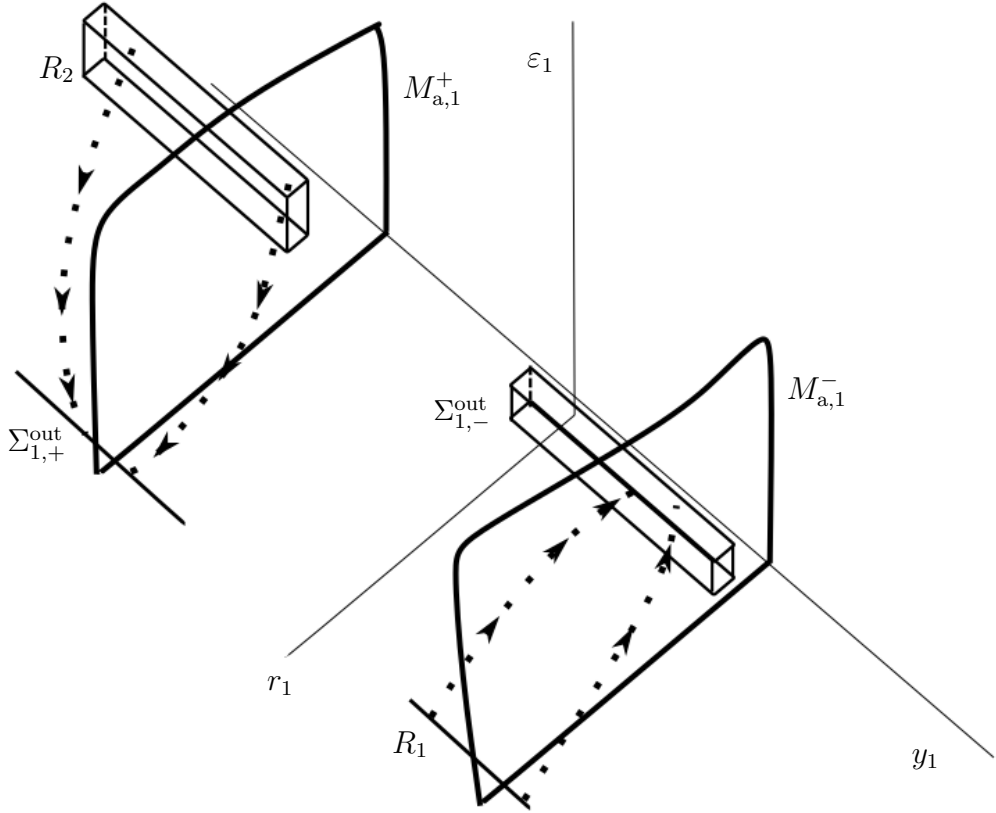
\begin{figure}[htbp]
        \centering
  		\begin{overpic}[width=0.8\textwidth]{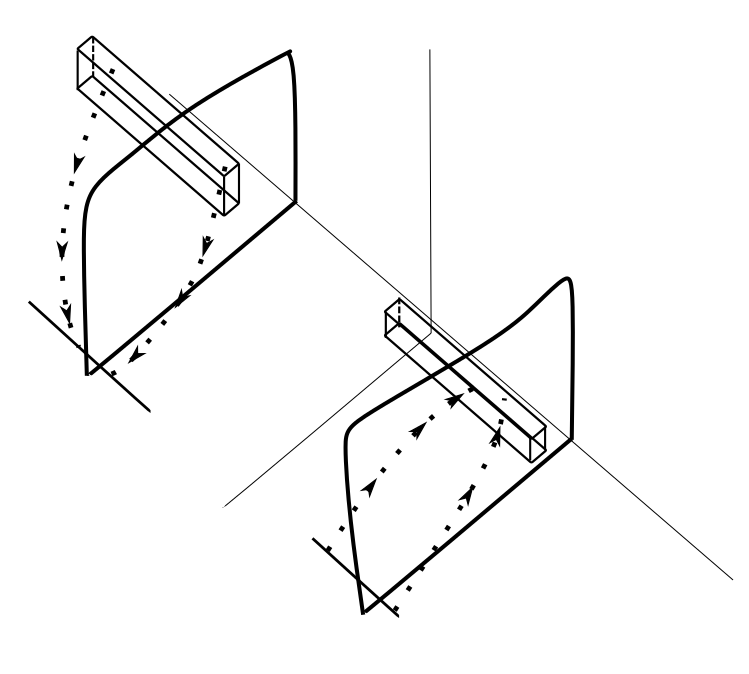}            
                \put(44,45){\small $\Sigma_{1,-}^{\textnormal{out}}$}
                \put(2,42){ $\Sigma_{1,+}^{\textnormal{out}}$}
          \put(39,13){ $R_1$}   
          \put(5,78){ $R_2$}  
           \put(89,12){ $y_1$} 
            \put(52,79){ $\epsilon_1$} 
             \put(27,19){ $r_1$}
             \put(40,76){  $M_{\txta,1}^{+}$} 
             \put(77,47){ $M_{\txta,1}^{-}$}  
        \end{overpic}
		\caption{Illustration of the dynamics in chart $K_1$ in $(r_1,y_1,\epsilon_1)$-space. Note that, due to the relation $h_1 = r_1 h$, we can interpret $l_{\pm}(\epsilon_1, h_1)$ as functions of $(\epsilon_1, r_1)$ and sketch its graphs, the center-stable manifolds $M_{\txta,1}^{\pm}$, as in the figure above. The figure shows trajectories (dotted lines) starting in $R_1$ up to reaching $\Sigma_{1,-}^{\textnormal{out}}$ in the vicinity of $M_{\txta,1}^{-}$ and trajectories starting in $R_2$ up to reaching the vicinity of $\Sigma_{1,+}^{\textnormal{out}}$ close to $M_{\txta,1}^{+}$.}
        \label{fig:K1}
\end{figure}

We can make the following statement about the transition time from $R_1$ to 
$\Sigma_{1,-}^{\textnormal{out}}$ which will be crucial for estimates on the 
contraction close to $M_{\txta,1}^-$. Define $\gamma := 2 \left| \lambda -1\right| 
+ \left| \lambda \right|$ and assume without loss of generality that $ \nu < \frac{1}{8}$.

\begin{lemma} 
\label{transitiontime_K1}
The transition time $N$ of system~\eqref{K1dynamics} from a point $p = (\rho, y_1, \delta/4, \nu)$ 
in $R_1$ to the point $\Pi_{1,-}(p)$ in $\Sigma_{1,-}^{\textnormal{out}}$ satisfies
$$ N \geq \frac{1}{17 \gamma } \frac{1}{\nu \delta} \,.$$
\end{lemma}
\begin{proof}
Let $(\epsilon_1(n))_{n\in \mathbb{N}}$ denote the trajectory starting at $\epsilon_1(0) = \delta/4 $ with
$$ \epsilon_1(n+1) =  \epsilon_1(n) (1 - h_1(n)F_1(y_1(n), \epsilon_1(n)))^{-2}.$$
We can show by induction that for all $n \in \mathbb{N}$ such that $\epsilon_1(n) \leq \delta$ we have
\begin{equation} \label{claim1}
\epsilon_1(n) \leq  \frac{\delta}{4} + n \left( 2\gamma\nu \delta^2 +  
f(\nu, \delta) \right)\,,
\end{equation}
where $f(\nu, \delta) = \mathcal{O}(\nu \delta^3)$ does not depend on $n$.
In more detail, we observe that
\begin{align*}
h_1(n)F_1(y_1(n), \epsilon_1(n)) &\leq  h_1(n) \left[ 1 - (-1 + \left| \lambda - 
1 \right| \delta)^2 + \lambda \epsilon_1(n)) \right] \\
& \leq \nu \left[ 2\left| \lambda - 1 \right| \delta + \left| \lambda \right| \delta) 
- ( \lambda - 1 )^2 \delta^2 \right] \\
& = \nu \gamma \delta - \nu ( \lambda - 1 )^2 \delta^2\,.
\end{align*}
Hence, we conclude with a first order Taylor approximation that for some $g(\nu, \delta) = \mathcal{O}(\nu \delta^2)$ we have
\begin{align*}
\epsilon_1(1) &\leq \frac{\delta}{4} \left( 1 + \nu \gamma \delta + 
g(\nu, \delta) \right)^2  
= \frac{\delta}{4} + \frac{\gamma}{2}\nu \delta^2  + \frac{\delta}{4} \left(2 g(\nu, \delta) + g(\nu, \delta)^2 + \nu^2 \gamma^2 \delta^2 + 2 g(\nu,\delta) \nu \gamma \delta \right)
\\
& \leq \frac{\delta}{4} + 2\gamma\nu \delta^2 + f(\nu, \delta)
\,,
\end{align*}
where  $f(\nu, \delta) =  \delta (2 g(\nu, \delta) + \mathcal{O}(\nu^2 \delta^2) )= \mathcal{O}(\nu \delta^3)$. Similarly, the step from $n$ to $n +1$ can be written as
\begin{align*}
\epsilon_1(n+1) &\leq   \epsilon_1(n) \left( 1 + \nu \gamma \delta + 
g(\nu, \delta) \right)^2 \leq
\epsilon_1(n) + 2 \gamma \nu  \delta  \epsilon_1(n) 
+  f(\nu,\delta)\,, \\
&\leq  \frac{\delta}{4} + n \left( 2\gamma\nu \delta^2 +  
f(\nu, \delta) \right) +  2 \gamma \nu \delta^2 
+ f(\nu,\delta)\,, \\
&= \frac{\delta}{4} + (n+1) \left( 2 \gamma \nu \delta^2 +  
f(\nu,\delta) \right)\,.
\end{align*}
This shows \eqref{claim1} for all $n \in \mathbb{N}$ such that $\epsilon_1(n) \leq \delta$.
We can rewrite the right hand side of~\eqref{claim1}, using a geometric series, as
$$\frac{\delta}{4} + n \left( 2\gamma\nu \delta^2 +  
f(\nu, \delta) \right) 
= \frac{\frac{\delta}{4}}{1 - n\left( 8 \gamma \nu \delta + \tilde{f}(\nu, \delta)  \right)}\,,$$
where $ \tilde{f}(\nu, \delta) = \mathcal{O}\left(\nu \delta^2\right)$.
By definition of the transition time $N$ we have $ \epsilon_1(N) \geq \delta(1 - 2 \nu)$. 
Hence, we deduce that
$$  \delta(1 - 2 \nu) \leq \frac{\frac{\delta}{4}}{1 - N ( 8 \gamma \nu \delta + \tilde{f}(\nu, \delta)  )}\,,$$
and therefore
$$  \delta \left(1 - 2 \nu- \frac{1}{4} \right)  \leq N \delta(1 - 2 \nu)( 8 \gamma \nu \delta 
+ \tilde{f}(\nu, \delta)  ) \,.$$
Finally, for $\delta$ sufficiently small and due to $\nu < \frac{1}{8}$, this leads to
$$N \geq \frac{\frac{\delta}{2}}{\delta(1 - 2 \nu)( 8 \gamma \nu \delta + \tilde{f}(\nu, \delta)  )}
\geq \frac{1}{17 \gamma \nu \delta } \,,$$
which concludes the proof.
\end{proof}

In addition to the first passage moving up the sphere, we already anticipate that 
for $\lambda<1$ trajectories eventually re-enter $K_1$ from 
$K_2$. With more precision to be added after the analysis in chart $K_2$, we define
\begin{align*}
\Sigma_{1,+}^{\textnormal{in}} &:=  \{(r_1, y_1, \epsilon_1, h_1) \in D_1 \,:\, 
\left| \epsilon_1 - \delta \right| \text{ small}, \left| r_1 - \frac{\rho}{2} \right| \text{ small}, \left| h_1 - \frac{\nu}{2} \right| \text{ small} \}\,, \\
\Sigma_{1,+}^{\textnormal{out}} &:= \{(r_1, y_1, \epsilon_1, h_1) \in D_1 \,:\,
 r_1 = \rho, \ h_1 = \nu, \epsilon_1 = \delta/4, \ y_1 > 0 \}\,,
\end{align*}
and denote by $\Pi_{1,+}: \Sigma_{1,+}^{\textnormal{in}} \to \Sigma_{1,+}^{\textnormal{out}}$ 
the map that sends points in $\Sigma_{1,+}^{\textnormal{in}}$ along a trajectory 
of~\eqref{K1dynamics} to the point of this trajectory, which is closest to 
$\Sigma_{1,+}^{\textnormal{out}}$. Note that $\Pi_{1,+}$ is well-defined sufficiently 
close to $M_{\txta,1}^+$ according to Proposition~\ref{Invariance_Prop}. In more detail, for $\beta_1^+$ and $\hat \beta_1^+$ to be determined more precisely in the analysis of chart $K_2$, there is
\begin{equation} \label{R2}
R_2 := \left\{(r_1, y_1, \epsilon_1, h_1) \in \Sigma_{1,+}^{\textnormal{in}} \,;\, 
1- \hat \beta_1^+\leq  y_1  \leq 1 + \beta_1^+ \right\}
\end{equation}
such that $M_{\txta,1}^+ \cap \Sigma_{1,+}^{\textnormal{in}} \subset R_2$ and $\Pi_{1,+}$ is well-defined on $R_2$; see also Figure~\ref{fig:K1}. Of course,
a completely analogous result for the passage time as stated in 
Proposition~\ref{transitiontime_K1} also holds for the map $\Pi_{1,+}$. We can use the 
lower bounds on the transition times to find the following lower bounds for the contraction 
rates of $\Pi_{1,-}|R_1$ and $\Pi_{1,+}|R_2$.

\begin{proposition} 
\label{contraction_K1}
There are constants $K_1, K_2 > 0$ such that for any $c$ with $0 < c < \nu = \rho h $
\begin{enumerate}
\item  the map $\Pi_{1,-}|R_1$ is a contraction (in the $y_1$-direction) with a rate stronger than
 $$ K_1 (1- c)^{\frac{1}{17 \gamma } \frac{1}{\nu \delta} }.$$
\item the map $\Pi_{1,+}|R_2$ is a contraction (in the $y_1$-direction) with a rate stronger than
 $$ K_2 (1- c)^{\frac{1}{17 \gamma } \frac{1}{\nu \delta} }.$$
\end{enumerate}
\end{proposition}

\begin{proof}
The statement about $\Pi_{1,-}$ follows from Lemma~\ref{transitiontime_K1} and the fact 
that the stable eigenvalue at the fixed points in $S_{\txta,1}^- \subset M_{\txta,1}^-$ 
is given by $1 - 2h_1 \geq 1 - \nu$, in combination with standard perturbation arguments.

The estimate for $\Pi_{1,+}$ uses the symmetry of system~\eqref{K1dynamics} with respect 
to the dynamics around $M_{\txta,1}^-$ and $M_{\txta,1}^+$: the transition time from 
$\Sigma_{1,+}^{\textnormal{in}}$ to $\Sigma_{1,+}^{\textnormal{out}}$ is of the same order 
as the transition time from $\Sigma_{1,-}^{\textnormal{in}}$ to $\Sigma_{1,-}^{\textnormal{out}}$, 
and the eigenvalues at $S_{\txta,1}^+$ are the same as at  $S_{\txta,1}^-$.
\end{proof}

\subsection{Dynamics in the scaling chart $K_2$} 
\label{secK2}

We turn to analyzing the dynamics in the scaling chart $K_2$ in order to understand the 
behaviour of trajectories past the origin. The chart $K_2$ covers the upper part of the 
sphere, where we can desingularize with respect to $\epsilon$. Recall from~\eqref{kappa12d} 
that the change of coordinates from $K_1$ to $K_2$ is given by
$k_{12}: K_1 \to K_2$ with
\begin{equation*}
x_2 = - \epsilon_1^{-1/2}, \quad y_2 =\epsilon_1^{-1/2} y_1, \quad 
r_2 = \epsilon_1^{1/2} r_1, \quad h_2 = \epsilon_1^{1/2} h_1\,,
\end{equation*}
It becomes clear from this transformation that the set of interest can be restricted to
$$ D_2 := \bigg\{(x_2,y_2,r_2,h_2)  \in \mathbb{R}^4 \, :\, \delta^{1/2}\frac{\rho}{2}
 \leq r_2 \leq \delta^{1/2}\rho, \  \delta^{1/2}\frac{\nu}{2} \leq h_2 \leq \delta^{1/2}\nu  \bigg\}\,.$$
First of all, we need to make sure that $\kappa_{1,2} \left(\Pi_{1,-} \left( R_1\right) \right) 
\subset \Sigma_2^{\textnormal{in}}$ for the entering set $\Sigma_2^{\textnormal{in}}$. 
From the analysis in $K_1$ we derive that this is satisfied for
\begin{align} \label{Sigma_2_in}
\Sigma_2^{\textnormal{in}} := \bigg\{(x_2, y_2,r_2, h_2) \in D_2 \, :\, &
 -(\delta(1 - 2 \nu))^{-1/2} \leq  x_2 \leq -\delta^{-1/2}, \nonumber \\
 &\delta^{-1/2}(-1 - \beta_2(\lambda)) \leq y_2 \leq \delta^{-1/2}(-1 + \hat \beta_2)  \bigg\}\,,
\end{align}
where
\begin{equation} \label{betas2}
\hat \beta_2 := \left| \lambda - 1\right| \delta\,, \quad \beta_2(\lambda):= 
\begin{cases} \frac{\lambda}{8} \delta &\text{ if } 0 < \lambda < 1 \\
\frac{2 \lambda -1}{4} \delta &\text{ otherwise.}
\end{cases}  
\end{equation}
We derive the desingularized equations and thereby justify the choice of blow-up 
in $h$. Observe that $\tilde r_2 = r_2$ since $\tilde \epsilon = \epsilon$ and 
$ \epsilon = r_2^2$. Similarly, we have $\tilde h_2 = h_2$. Furthermore observe that
$$ \tilde y_2 = \frac{\tilde y}{\tilde r_2} = \frac{y + \epsilon h}{\tilde r_2} = 
\frac{r_2 y_2 + r_2^2 h_2 r_2^{-1}}{ r_2} = y_2 + h_2.$$
In addition to that, we obtain
$$ \tilde x_2 = \frac{\tilde x}{\tilde r_2} =  \frac{r_2 x_2 + 
h_2 r_2^{-1}(r_2^2x_2^2 - r_2^2 y_2^2 + \lambda r_2^2}{ r_2} = x_2 + 
h_2(x_2^2 - y_2^2 + \lambda).$$
Hence, summarising, the dynamics in chart $K_2$ are given by iterating the map
\begin{align} \label{K2_discrete}
\tilde x_2 &= x_2 + h_2(x_2^2 - y_2^2 + \lambda)\,, \nonumber \\
\tilde y_2 &=  y_2 + h_2\,, \nonumber \\
\tilde r_2 &= r_2\,, \nonumber \\
\tilde h_2 &= h_2\,.
\end{align}
The transition areas from $K_2$ to another chart depend on $\lambda$. 
For $\lambda < 1$, we will return to chart $K_1$. Recall from~\eqref{kappa21d} 
that the change of coordinates $k_{21}: K_2 \to K_1$ is given by
\begin{equation*}
\epsilon_1 = x_2^{-2}, \quad  y_1 = -x_2^{-1} y_2, \quad r_1 = - x_2 r_2, \quad h_1 = - x_2 h_2,
\end{equation*}
We need to choose $\Sigma_{2,\txta }^{\textnormal{out}}$ and the cuboid $R_2$ in $\Sigma_{1,+}^{\textnormal{in}}$ in chart $K_1$ (see \eqref{R2}) such that, firstly, trajectories starting in $\Sigma_2^{\textnormal{in}}$ reach $\Sigma_{2,\txta }^{\textnormal{out}}$ and, secondly,
$k_{21}(\Sigma_{2,\txta }^{\textnormal{out}}) \subset R_2 $.  It turns out (see proof of 
Proposition~\ref{propK2} for the first criterion) that a suitable 
choice is given by 
\begin{align}\label{Sigma_2a_out}
\Sigma_{2,\txta }^{\textnormal{out}} :=  \bigg\{(x_2, y_2,r_2, h_2) \in D_2 \, 
:\,&  - \delta^{-1/2} - \frac{h_2}{2} \leq x_2 \leq - \delta^{-1/2} + \frac{h_2}{2}, \nonumber \\
& \delta^{-1/2} ( 1-  \hat \beta_2^+)\leq y_2 \leq \delta^{-1/2}(1 + \beta_2^+) \bigg\}\,,
\end{align}
where we define $\beta_2^+ : = \frac{\left|\lambda  + 1\right|}{2}  \delta$ and 
$\hat \beta_2^+ : = \frac{\left|\lambda\right|  + 1}{2}  \delta$; see also Figure~\ref{fig:transcrit}. The second criterion is then satisfied by adapting $\Sigma_{1,+ }^{\textnormal{in}}$  in the $(\epsilon_1, r_1, h_1)$-components accordingly via $k_{21}$ and choosing, for example, $\beta_1^+ : = \frac{3\left|\lambda  + 1\right|}{4}  \delta$ and 
$\hat \beta_1^+ : = \frac{3(\left|\lambda\right|  + 1)}{4}  \delta$ in the definition of $R_2$.

For $\lambda > 1$, we set the area of exit as
\begin{align}\label{Sigma_2e_out}
\Sigma_{2,\txte }^{\textnormal{out}} :=  \{(x_2, y_2,r_2, h_2) \in D_2 \, :\,& 
\delta^{-1/2}  \leq x_2 \leq  \delta^{-1/2} + h_2(\lambda + \delta^{-1})\,, \nonumber \\
& 0 \leq y_2 < \Omega(\lambda)\delta^{-1/6} \}\,,
\end{align} 
where $\Omega(\lambda) > 0$ is a constant for fixed $\lambda$; see also Figure~\ref{fig:transcrit}. In the situation of 
continuous time, the $y_2$-component in $\Sigma_{2,\txte }^{\textnormal{out}}$ can be 
bounded by a constant independent from $\delta$, by using the Riccati 
equation \cite[Proposition~2.3]{ks2011}. As we do not have such a tool in the case of maps, 
we give an estimate from a first oder expansion in $h$ of the iterated maps (see proof of 
Proposition~\ref{propK2}).

Let us denote the sequence induced by iterating \eqref{K2_discrete} for some initial 
condition $(x_0,y_0)$ as $(x_2(n), y_2(n))$ for $n \in \mathbb N$, and call such a 
sequence a trajectory. As for continuous time, the special case is the canard problem, 
i.e. when $\lambda =1$. In this case, for any $c=x_0 = y_0 \in \mathbb{R}$ the system 
of maps has the obvious solution $\gamma_2^c(n)$ with $x_2(n) = y_2(n) = c + n h_2$. 
For $\lambda \neq 1$ we can make the following direct observations about the dynamics 
of the maps, where we recall that $\nu := \rho h < \delta$, in particular $\nu < \frac{1}{8}$.

\begin{figure}[H]
\centering
\begin{subfigure}{.5\textwidth}
        \centering
  		\begin{overpic}[width=.9\textwidth]{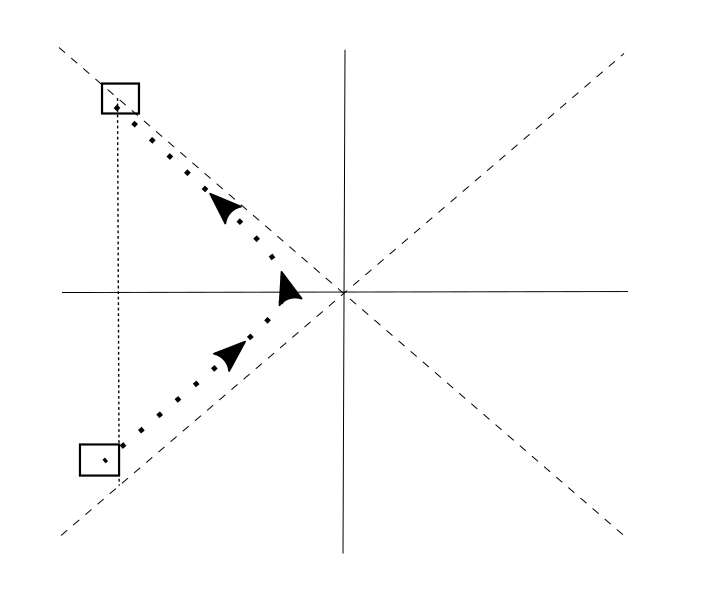}            
                \put(6,25){ $\Sigma_{2}^{\textnormal{in}}$}
                \put(13,77){ $\Sigma_{2,\txta }^{\textnormal{out}}$}
                \put(17,39){\scriptsize $- \delta^{-1/2}$}
                \put(47,82){\scriptsize $y_2$}
                \put(91,43){\scriptsize $x_2$}
        \end{overpic}
        \caption{$\lambda < 0$}
  \label{lambda_negative}
\end{subfigure}%
\begin{subfigure}{.5\textwidth}
        \centering
  		\begin{overpic}[width=.9\textwidth]{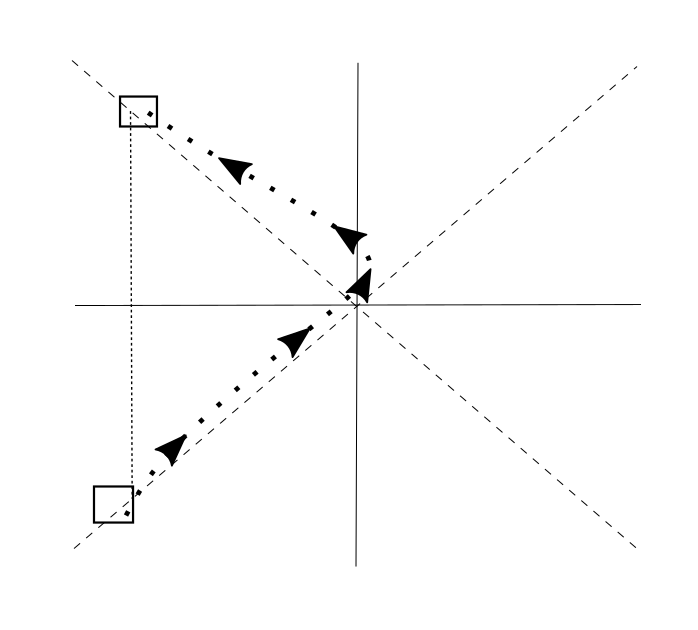}            
 \put(9,23){ $\Sigma_{2}^{\textnormal{in}}$}
                \put(15,80){ $\Sigma_{2,\txta }^{\textnormal{out}}$}
                \put(19,41){\scriptsize $- \delta^{-1/2}$}
                \put(49,84){\scriptsize$y_2$}
                \put(95,44){\scriptsize $x_2$}
        \end{overpic}
        \caption{$0 < \lambda < 1$}
  \label{lambda_small}
\end{subfigure}%
\newline
\begin{subfigure}{.5\textwidth}
        \centering
  		\begin{overpic}[width=.9\textwidth]{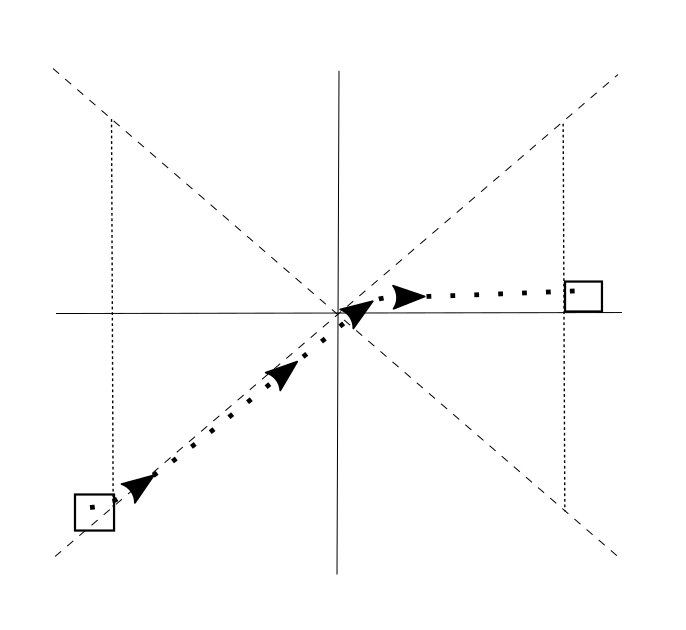}            
                \put(6,22){ $\Sigma_{2}^{\textnormal{in}}$}
                \put(84,53){ $\Sigma_{2,\txte }^{\textnormal{out}}$}
                \put(17,38){\scriptsize $- \delta^{-1/2}$}
                \put(72,38){\scriptsize $ \delta^{-1/2}$}
                \put(48,84){\scriptsize $y_2$}
                \put(91,43){\scriptsize $x_2$}
        \end{overpic}
        \caption{$\lambda > 1$}
  \label{lambda_large}
\end{subfigure}
\caption{Typical trajectories (dotted lines) in the $K_2$ chart, i.e.~solutions of~\eqref{K2_discrete} 
starting in $\Sigma_{2}^{\textnormal{in}}$ for different values of 
$\lambda$, as ascertained in the proof of Proposition~\ref{propK2}. For $\lambda < 0$ (a) the $y_2$-axis is not crossed, as opposed to the case $0 < \lambda < 1$ (b), but in both situations the rectangle $\Sigma_{2,\txta }^{\textnormal{out}}$ is reached close to the diagonal (dashed lines), specifically $(x_2,y_2)=(- \delta^{-1/2}, - \delta^{-1/2})$.  For $\lambda > 1$ (c), trajectories reach the rectangle $\Sigma_{2,\txte }^{\textnormal{out}}$ close to $\{x_2 = \delta^{-1/2}\}$ above the $x_2$-axis. }
\label{fig:transcrit}
\end{figure}
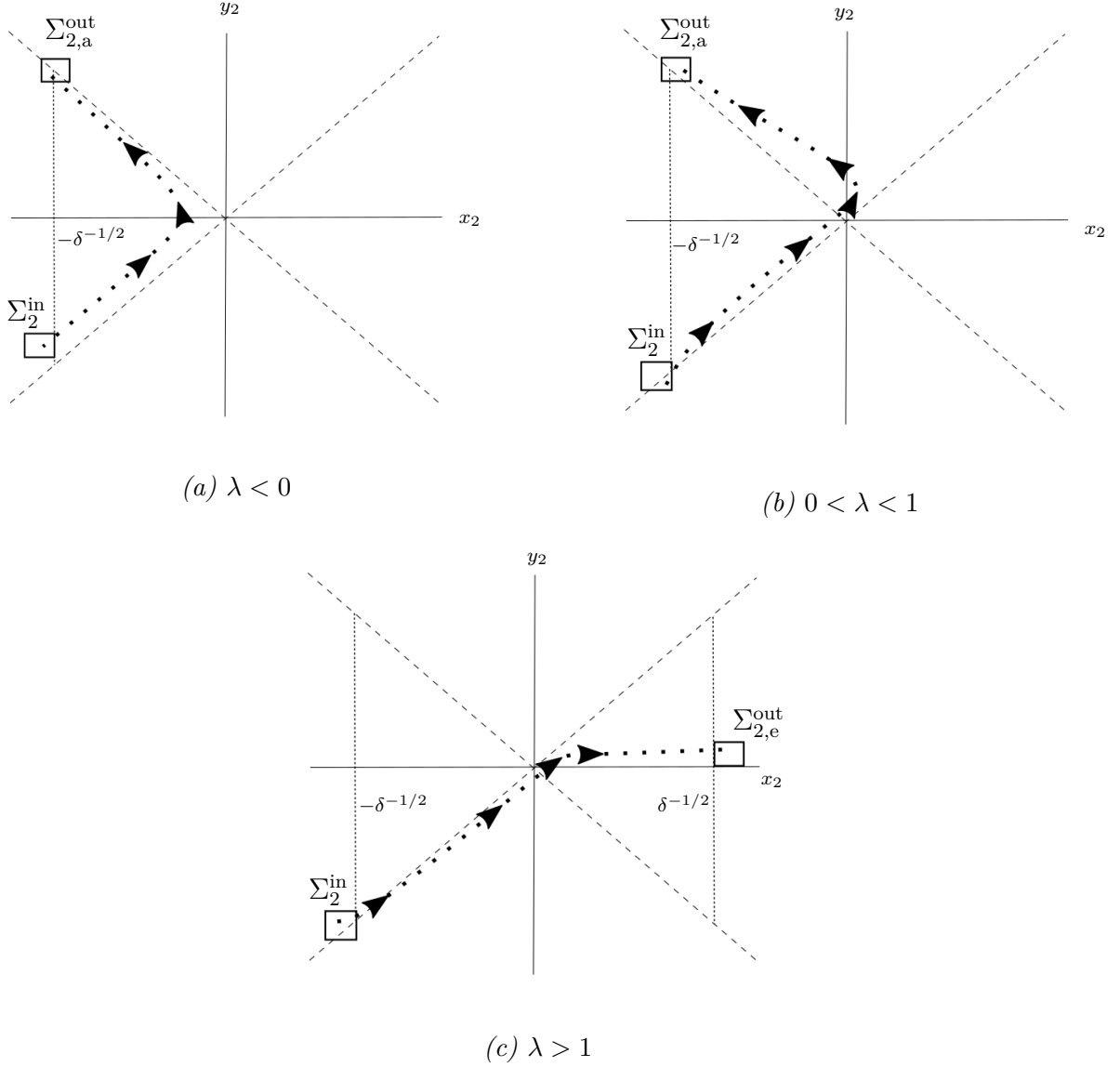

\begin{proposition} 
\label{propK2}The following results hold:
\begin{enumerate}
\item[(P1)] If $\lambda < 1$, every trajectory starting in $\Sigma_2^{\textnormal{in}}$ 
passes through $\Sigma_{2,\txta }^{\textnormal{out}}$. 
\item[(P2)] If $\lambda > 1$, every trajectory starting in $\Sigma_2^{\textnormal{in}}$ 
passes through  $\Sigma_{2,\txte }^{\textnormal{out}}$.
\end{enumerate}
\end{proposition}

The proof of this proposition is based on a couple of lemmas, which are shown in the 
following. We divide the diagonals $\{x=y\}$ and $\{x = -y\}$ into the subsets
$$S_{\txta,2}^-:=\{(x,y) \in \mathbb{R}^2 \,:\, y \leq 0, \ x = y\}\,, \quad S_{\txta,2}^{+} 
:= \{(x,y) \in \mathbb{R}^2 \,:\, y \geq 0, \ x = -y\} $$
and
$$S_{\txtr,2}^-:=\{(x,y) \in \mathbb{R}^2 \,:\, y \leq 0, \ x = -y\}\,, \quad S_{\txtr,2}^{+} 
:= \{(x,y) \in \mathbb{R}^2 \,:\, y \geq 0, \ x = y\}\,. $$
Furthermore, we write as a shorthand $x_{2,n} = x_2(n), y_{2,n} = y_2(n)$ for $n \in \mathbb{N}$ 
and investigate the behaviour of the trajectories
\begin{align*}
 y_{2,n+1} &= y_{2,n}  + h_2\,, \\
 x_{2,n+1} &= x_{2,n} + h_2 \lambda + h_2 (x_{2,n}^2 - y_{2,n}^2)
\end{align*} 
for different values of $\lambda$. In fact, even for $\lambda<1$, there are subtle differences
in the paths of trajectories (see Figure~\ref{fig:transcrit}).

\begin{lemma} 
\label{lambda01}The following cases occur for $\lambda<1$:
\begin{itemize}
\item Let $ 0 < \lambda < 1$ and $\delta$ be sufficiently small. Then any trajectory 
starting in $\Sigma_2^{\textnormal{in}}$ is strictly increasing in $x$ as long as 
$x_{2,n}, y_{2,n} < 0$, and will be above the diagonal $\{x = y\}$ at a certain point 
of time and stay there forever afterwards. In particular, if $(x_{2,0}, y_{2,0})$ in 
$\Sigma_2^{\textnormal{in}}$ with $y_{2,0} < x_{2,0} < 0$, there exists $n^* \in \mathbb{N}$ 
such that $x_{2,n^*}  \leq y_{2,n^*} < 0$ and
$$\frac{n^* h_2}{n^*h_2 + \frac{\delta^{1/2}}{8}} \geq \lambda\,.$$
\item If $\lambda \leq 0$, any trajectory starting in $\Sigma_2^{\textnormal{in}}$ is strictly 
increasing in $x$ as long as $x_{2,n}, y_{2,n} < 0$ and will be above the diagonal $\{x = y\}$ 
at any point of time.
\end{itemize}
\end{lemma}

\begin{proof}
We start with the case $ 0 < \lambda < 1$. Consider an initial condition $(x_{2,0}, y_{2,0})$ 
in $\Sigma_2^{\textnormal{in}}$ with $y_{2,0} < x_{2,0} < 0$, i.e., below $S_{\txta,2}^-$. We 
obviously have $x_{2,1} < x_{2,0} + \lambda h_2$. Furthermore, observe that
\begin{align*}
x_{2,0}^2 - y_{2,0}^2 + \lambda &\geq  \delta^{-1}  - ( \delta^{-1/2}(-1 - \beta_2(\lambda)))^2 + \lambda \\
&= \delta^{-1}  - \delta^{-1} - \frac{1}{4} \lambda - \frac{1}{64} \lambda^2 \delta + \lambda \\
& = \frac{3}{4} \lambda - \frac{1}{64} \lambda^2 \delta \geq \frac{1}{2} \lambda.
\end{align*}
Hence, $x_{2,1} \geq  x_{2,0} + \frac{\lambda}{2}  h_2$. Either we already have $ x_{2,1}  
\leq y_{2,1} < 0$. If not, we can infer from the facts $ 0 > y_{2,1} > y_{2,0}$, $ 0 > x_{2,1} 
> x_{2,0}$ and $  x_{2,1} - y_{2,1} < x_{2,0} - y_{2,0}$  that $y_{2,1}^2 - x_{2,1}^2  <  
y_{2,0}^2 - x_{2,0}^2$. Hence, we have $x_{2,0} + \lambda h_2 < x_{2,2} < x_{2,0} + 2 \lambda h_2$ 
and obviously $ y_{2,2} = y_{2,0} + 2 h_2$. Therefore, we see inductively that for $0 > x_{2,n}  
> y_{2,n}$ both sequences are increasing and we either already have $ x_{2,n+1}  \leq y_{2,n+1} < 0$ or 
\begin{align*}
x_{2,n+1} - y_{2,n+1} &<  x_{2,0} - y_{2,0} -(n+1)(1- \lambda)h_2 < \frac{\lambda}{8} 
\delta^{1/2} -(n+1)(1- \lambda)h_2 \\
&= \lambda \left( (n+1) h_2 + \frac{\delta^{1/2}}{8}  \right) - (n+1) h_2.
\end{align*}
Thus, we can conclude that $x_{2,n^*}  \leq y_{2,n^*} < 0$ for some $n^*$ such that $\lambda 
\leq \frac{n^* h_2}{n^*h_2 + \frac{\delta^{1/2}}{8}}$, if $\delta$ is chosen small enough such 
that $\frac{\delta^{-1/2}}{\delta^{-1/2} + \frac{\delta^{1/2}}{8}} > \lambda$. Namely, if there 
was $\hat n \in \mathbb{N}$ such that $0 > x_{2,n}  > y_{2,n}$ for all $n < \hat n$ and $x_{2,\hat n} 
> 0$, we would have, for $h_2, \delta$ small enough, that $\frac{\hat n -1}{\hat n} \geq \lambda$ 
and would obtain 
$$\frac{(\hat n-1) h_2}{  (\hat n -1) h_2 + \frac{\delta^{1/2}}{8}} \geq \frac{\lambda 
\hat n h_2}{ \lambda \hat n h_2 + \frac{\delta^{1/2}}{8}} >  \frac{\delta^{-1/2}}{\delta^{-1/2}  
+ \frac{\delta^{1/2}}{8}} > \lambda\,,$$
which is a contradiction, since this would imply $x_{2, \hat n -1}  < y_{2,\hat n -1} $ with the above.

Assume now that, at time $n \in \mathbb{N}$, the trajectory is above the diagonal $\{x = y\}$, 
i.e., $x_{2,n} \leq y_{2,n}$. In particular, this covers the case when $x_{2,n} \leq y_{2,n} < 0$, when 
the trajectory lies above $S_{\txta,2}^-$, which is relevant for the initial data.
We have  $y_{2,n+1} = y_{2,n}  + h_2$ and 
 $x_{2,n+1} \geq x_{2,n} + h_2 \lambda$.
If $x_{2,n} = y_{2,n}$, then obviously $ x_{2,n+1} <  y_{2,n+1}$. If $x_{2,n} < y_{2,n}$, we observe that 
\begin{equation} \label{x2biggery2}
 x_{2,n+1} \geq y_{2,n+1} \quad \text{iff} \quad 1 + \frac{h_2(1- \lambda)}{y_{2,n}-x_{2,n}} \leq - h_2 ( x_{2,n} + y_{2,n} )\,.
\end{equation}
Hence, since $ h_2 \left| x_{2,n} + y_{2,n} \right| <  2 (1 - 2 \nu)^{-1/2} \delta^{-1/2} \delta^{1/2} 
 \nu < 1 $, we have $ x_{2,n+1} <  y_{2,n+1}$ and the argument goes on inductively. We can also see that, 
for $x_{2,n} \leq y_{2,n} < 0$, the trajectories stay close to $S_{\txta,2}^-$ since, if $ x_{2,n}^2 - 
y_{2,n}^2 > 1 - \lambda$, we have that $ y_{2,n+1} - x_{2,n+1} < y_{2,n} - x_{2,n}$. This concludes the
proof of the first statement.

Next, we consider the case $\lambda \leq 0$. Again, assume that at time $n \in \mathbb{N}$ the trajectory is 
above the diagonal $\{x = y\}$, i.e. $x_{2,n} \leq y_{2,n}$. In particular, this covers the case when 
$x_{2,n} \leq y_{2,n} < 0$ as relevant for the initial data.
If $x_{2,n} = y_{2,n}$, then obviously $ x_{2,n+1} <  y_{2,n+1}$. If $x_{2,n} < y_{2,n}$, we observe as 
before that~\eqref{x2biggery2} holds and that
$ h_2 \left| x_{2,n} + y_{2,n} \right| < 1 $. Hence, we have $ x_{2,n+1} <  y_{2,n+1}$ and the argument 
goes on inductively. Therefore trajectories stay above the diagonal. Furthermore, observe that $y_{2,0}^2 
\leq \left( \delta^{-1/2} \left( -1 - \frac{2\lambda - 1}{4}\delta \right) \right)^2 $ and therefore
$$ x_{2,0}^2 - y_{2,0}^2 + \lambda \geq \delta^{-1} - \delta^{-1} - \lambda + \frac{1}{2} - 
\frac{(2 \lambda -1)^2}{16} \delta + \lambda = \frac{1}{2} - \frac{(2 \lambda -1)^2}{16} \delta\,,$$
which is greater than $0$ for $\delta$ small enough, depending on $\lambda$. Hence, $x_{2,1} > x_{2,0}$. 
We show that $x_{2,n+1} > x_{2,n}$ as long as $x_{2,n} < y_{2,n} < 0$ by proving that $\xi_n := x_{2,n}^2 
- y_{2,n}^2 + \lambda > 0$ implies $x_{2,n+1}^2 - y_{2,n+1}^2 + \lambda > 0$. Assuming that $x_{2,n} < y_{2,n} 
< 0$ and $\xi_n > 0$ yields
\begin{align*}
x_{2,n+1}^2 - y_{2,n+1}^2 + \lambda &= (x_{2,n} + h_2 \xi_n)^2 - (y_{2,n} + h_2)^2 + \lambda \\
&= \xi_n + h_2^2(\xi_n^2 -1) + 2 h_2 ( \left| y_{2,n} \right| - \xi_n \left| x_{2,n} \right|)\,.
\end{align*}
From there, it is easy to observe that for $h_2$ small enough the claim follows.
\end{proof}

Although the argument is quite technical, the proof of the last lemma shows that the key
steps in the scaling chart involve the sign of the nonlinear term for the $x_2$-variable.
This idea can also be carried out in the case $\lambda>1$.

\begin{lemma} 
\label{lambda1big}
Let $\lambda > 1$ and $\delta$, $\nu$ sufficiently small. Then all trajectories starting 
in $\Sigma_{2}^{\textnormal{in}}$ are strictly increasing in $x$ as long as 
$ x_{2,n}, y_{2,n} < 0$, will be below the diagonal $\{x = y\}$ at a certain point of 
time and stay there forever afterwards. In particular, if $(x_{2,0}, y_{2,0})$ in 
$\Sigma_2^{\textnormal{in}}$ with $x_{2,0} < y_{2,0} < 0$, there is a $n^* \in \mathbb{N}$ 
such that $y_{2,n^*} \leq x_{2,n^*} < 0$ and
$$n^* h_2 \geq  \delta^{1/2}\,.$$
\end{lemma}

\begin{proof}
We consider two cases, trajectories below and above the diagonal. 
First, we assume that, at time $n \in \mathbb{N}$, the trajectory is below the diagonal 
$\{x = y\}$ so that $ y_{2,n} \leq x_{2,n}$. In particular, this covers the case when 
$ y_{2,n} \leq x_{2,n} < 0$, i.e., the trajectory lies below $S_{\txta,2}^-$, which is 
relevant for the initial data.  If $x_{2,n} = y_{2,n}$, then obviously $y_{2,n+1} < 
x_{2,n+1}$. If $y_{2,n} < x_{2,n}$, we observe similarly to~\eqref{x2biggery2} that 
$$ y_{2,n+1} \geq x_{2,n+1} \quad \text{iff} \quad 1 + \frac{h_2(\lambda 
- 1)}{x_{2,n}-y_{2,n}} \leq - h_2 ( x_{2,n} + y_{2,n} )\,.$$
Hence, since $ h_2 \left| x_{2,n} + y_{2,n} \right| <  (2 (1 - 2 \nu)^{-1/2}\delta^{-1/2} 
+ \frac{2 \lambda -1}{4} \delta^{1/2} )\delta^{1/2} \nu < 1$ due to $\delta \lambda \leq 1$ 
and $\nu < \frac{1}{8}$, we have $ y_{2,n+1} < x_{2,n+1} $ and the argument goes on inductively.

Moreover, we check that the sequences are increasing in this case. We consider an initial 
condition $(x_{2,0}, y_{2,0})$ in $\Sigma_2^{\textnormal{in}}$ with $y_{2,0} < x_{2,0} < 0$, 
i.e., below $S_{\txta,2}^-$. We obviously have $x_{2,1} < x_{2,0} + \lambda h_2$. Furthermore, observe that
\begin{align*}
x_{2,0}^2 - y_{2,0}^2 + \lambda &\geq \delta^{-1} - ( \delta^{-1/2}(-1 - \beta_2(\lambda)))^2 + \lambda \\
&= \delta^{-1} - \delta^{-1} - \lambda + \frac{1}{2} - \frac{(2 \lambda - 1)^2}{16} \delta + \lambda \\
& = \frac{1}{2} - \frac{(2 \lambda - 1)^2}{16} \delta\geq \frac{1}{4}\,,
\end{align*}
if $\delta$ is chosen sufficiently small in comparison to $\lambda$.
Hence, $x_{2,1} \geq  x_{2,0} + \frac{1}{4}  h_2$. Note in particular that if $y_{2,0} 
= \delta^{-1/2}(-1 - \beta_2(\lambda)) =: y^*$, we have $x_{2,0}^2 - y_{2,0}^2 + \lambda 
< \frac{1}{2}$.  From here, it is easy to observe that, for $\nu$ small enough, $  x_{2,1} 
- y_{2,1} < x_{2,0} - y^* $. This together with the fact that $ \left| y_{2,1} + x_{2,1} \right| 
< \left| y_{2,0} + x_{2,0} \right|$ yields $ x_{2,1}^2-y_{2,1}^2 + \lambda > \frac{1}{4}$. Hence, 
we have $ x_{2,2} \geq  \frac{1}{2} h_2$ and obviously $ y_{2,2} = y_{2,0} + 2 h_2$. Therefore, we 
see inductively that for $0 > x_{2,n}  > y_{2,n}$ both sequences are increasing.

As the second case, we consider a trajectory with initial condition $(x_{2,0}, y_{2,0})$ 
in $\Sigma_2^{\textnormal{in}}$ and $x_{2,0} < y_{2,0} < 0$, i.e., above $S_{\txta,2}^-$. 
Either we already have $ y_{2,1} \leq x_{2,1} < 0$. If not, we have $x_{2,1} >  x_{2,0} + 
\lambda h_2$ and obviously $ y_{2,1} = y_{2,0} +  h_2$. Therefore, we see inductively that 
for $x_{2,n}  < y_{2,n} <0$ both sequences are increasing and we either already have $ y_{2,n+1} 
\leq x_{2,n+1} < 0$ or 
\begin{align*}
x_{2,n+1} - y_{2,n+1} & >  x_{2,0} - y_{2,0} +(n+1)(\lambda -1)h_2 \\
& \geq \delta^{-1/2} - \delta^{-1/2}(1 - 2 \nu)^{-1/2} - (\lambda-1)\delta^{1/2} 
+(n+1)(\lambda -1)h_2 .
\end{align*}
Thus, we can conclude that $ y_{2,n^*} \leq x_{2,n^*} < 0$ for some $n^*$ such that 
$n^* h \geq \delta^{1/2}\left( 1+ \frac{(1 - 2 \nu)^{-1/2}-1}{\lambda -1} \right)$. 
Using $(1 - 2 \nu)^{-1/2} > 1$, the claim follows.
\end{proof}

Finally, we turn to the proof of Proposition~\ref{propK2}.

\begin{proof}[Proof of Proposition~\ref{propK2}] We distinguish three cases: (I) $ 0 < \lambda < 1$,
(II) $\lambda\leq 0$, and (III) $\lambda>1$. The case distinction is going to allow us to apply
each of the preliminary results obtained above.

\textit{Case (I) $0<\lambda<1$:}
From Lemma~\ref{lambda01} we know that trajectories starting in $\Sigma_2^{\textnormal{in}}$ 
will be above the diagonal $\{x = y\}$ at certain point of time and stay there forever afterwards. 
From this result and the fact that $y_{2,n}$ and $x_{2,n}$ are both strictly increasing uniformly 
as long as $\left|y_{2,n} \right| \leq \left|x_{2,n}\right|$, we can conclude that any such trajectory 
reaches a point $(x_{2,\tilde n}, y_{2,\tilde n})$, with $y_{2,\tilde n} > 0$, such that 
$y_{2,\tilde n}^2 > x_{2,\tilde n}^2$. We can conclude that there must be a minimal $n^* > \tilde n$ 
such that $y_{2,n^*}^2 > x_{2,n^*}^2 + \lambda$. Note that $(x_{2,n^*}, y_{2,n^*})$ lies between 
$ S_{\txta,2}^{+}$ and $ S_{\txtr,2}^{+}$. As long as this is the case
for $n > n^*$, we have $x_{2,n+1} < x_{2,n}$. Additionally, we observe that for $y_{2,n}^2 - x_{2,n}^2 > 
\lambda +1$ we have $y_{2,n+1} + x_{2,n+1} < y_{2,n} + x_{2,n}$. Hence, trajectories are rapidly 
approaching the vicinity of $S_{2,\txta }^{+}$. Similarly to~\eqref{x2biggery2}, we find that for any 
such $y_{2,n} > - x_{2,n} > 0$
$$ \left|x_{2,n+1} \right| \geq y_{2,n+1} \quad \text{iff} \quad 1 
+ \frac{h_2(1+\lambda)}{y_{2,n}-\left|x_{2,n}\right|} \leq  h_2 ( \left|x_{2,n}\right| + y_{2,n} )\,.$$
Hence, since $ h_2 \left| x_{2,n}\right| + y_{2,n}  <  (2 + \delta + h_2) \nu < 1 $ before hitting 
$\Sigma_{2,\txta }^{\textnormal{out}}$, we have $ \left|x_{2,n+1} \right| <  y_{2,n+1}$ and the argument goes on 
inductively before hitting $\Sigma_{2,\txta }^{\textnormal{out}}$. The fact that the trajectory will 
actually be located within $\Sigma_{2,\txta }^{\textnormal{out}}$ at a certain point of time can be 
inferred as follows: we observe from the above that $(x_{2,n},y_{2,n})$ satisfies $0 \leq y_{2,n}^2 
- x_{2,n}^2 \leq \lambda +1$ for large enough $n$. First, we can conclude that $x_{2,n} - x_{2,n+1} 
\leq h_2$. Hence, there is an $m \in \mathbb{N}$ such that $x_{2,m} \in [-\delta^{-1/2} - \frac{h_2}{2},
-\delta^{-1/2} +\frac{h_2}{2}]$. Therefore we have $y_{2,m} \geq \delta^{-1/2} - \frac{h_2}{2}$ and
$$  y_{2,m}^2 \leq \delta^{-1} + (\lambda +1) - h_2 \delta^{-1/2} + \frac{h_2^2}{4} \leq \delta^{-1} 
+ 2 \beta_2^+ \delta^{-1} + \delta^{-1} (\beta_2^+)^2 = \left(\delta^{-1/2}(1 + \beta_2^+)\right)^2. $$
Figure~\ref{fig:transcrit} (b) illustrates the behaviour of trajectories starting in 
$\Sigma_{2,\txta }^{\textnormal{in}}$ for $ 0 < \lambda <1$.

\textit{Case (II) $\lambda\leq 0$:} We know from Lemma~\ref{lambda01} that any trajectory starting 
in $\Sigma_2^{\textnormal{in}}$ is strictly increasing in $x$ as long as $x_{2,n}, y_{2,n} < 0$ and 
will be above the diagonal $\{x = y\}$ at any point of time. Analogously to before, we can conclude 
that there exists a minimal $n^* \in \mathbb{N}$ such that $y_{2,n^*} > 0$ and $y_{2,n^*}^2 > x_{2,n^*}^2 
+ \lambda$. Note that if $(x_{2,n^*}, y_{2,n^*})$ lies between $ S_{\txta,2}^{+}$ and $ S_{\txtr,2}^{+}$ 
and stays in this region for all $n > n^*$ before hitting $\Sigma_{\txta,2}^{\textnormal{out}}$, as for 
example for $ -1\leq \lambda \leq 0$, the arguments go exactly as before.
Otherwise we observe, symmetrically to before, that $(x_{2,n},y_{2,n})$ satisfies $0 \geq x_{2,n}^2 
- y_{2,n}^2 + \lambda \geq - 1$ for large enough $n$. Again, we infer that $x_{2,n} - x_{2,n+1} \leq h_2$ 
and conclude that there is an $m \in \mathbb{N}$ such that $x_{2,m} \in [-\delta^{-1/2} - \frac{h_2}{2},
-\delta^{-1/2} +\frac{h_2}{2}]$. Therefore we have $y_{2,m} \leq \delta^{-1/2} + \frac{h_2}{2}$ and for 
$\delta$ sufficiently small depending on $\lambda$
$$  y_{2,m}^2 \geq \delta^{-1} - h_2 \delta^{-1/2} + \frac{h_2^2}{4} + \lambda \geq \delta^{-1} + \lambda  
+ \frac{(\left| \lambda \right| + 1)^2}{4} \delta - 1 = \left(\delta^{-1/2}(1 - \hat \beta_2^+)\right)^2. $$
Figure~\ref{fig:transcrit}~(a) illustrates the behaviour of trajectories starting in 
$\Sigma_{2,\txta }^{\textnormal{in}}$ for $ \lambda < 0$.

\textit{Case (III) $\lambda>1$:} We can conclude from Lemma~\ref{lambda1big} that trajectories starting 
in $\Sigma_2^{\textnormal{in}}$ will be below $S_{\txta,2}^-$ at a certain point of time and stay below 
the diagonal $\{x = y\}$ forever afterwards. From that and the fact that $y_{2,n}$ is strictly increasing 
for all time, we can conclude that any such trajectory will reach a point $(x_{2,n^*}, y_{ 2, n^*})$ with 
$ x_{2,n^*}> y_{2,n^*} > 0$. Then the trajectory will increase its distance from the diagonal in each time 
step by
$$  h_2 (\lambda -1) + h_2 (x_{2,n}^2 - y_{2,n}^2)\,.$$
Let us take the largest $n$ such that $ x_{2,n}> 0 
\geq y_{2,n} $. It is now obvious that there is an $m \in \mathbb{N}$ such that $\delta^{-1/2}  \leq x_{2,n+m} \leq  \delta^{-1/2} + h_2(\lambda + \delta^{-1}) $. We give an upper bound for $y_{2,m+n}$ by expanding $x_n$ up to 
$h_2^3$, which is the first order estimate in this case:
\begin{align*}
(1 + 2\nu) \delta^{-1/2} &\geq x_{2,m +n}  > mh_2 \lambda  + 
\left(\sum_{k=1}^{m-1} k^2 \right) (\lambda^2 -1) h_2^3 \\ &\geq y_{2,m+n} \lambda  + 
\frac{1}{6} (\lambda^2 - 1)y_{2,m+n}(y_{2,m+n} -h_2)(2y_{2,m+n} -h_2) \\
&\geq \lambda y_{2,m+n} + \frac{1}{6} (\lambda^2 - 1)y_{2,m+n}^3\,.
\end{align*}
Hence, we conclude that $y_{2,m+n} = \mathcal{O}(\delta^{-1/6})$ and $(x_{2,n+m}, y_{2,m+n}) \in 
\Sigma_{2,\txte }^{\textnormal{out}}$.
Figure~\ref{fig:transcrit} (c) illustrates such a trajectory. 
\end{proof}

\subsection{Dynamics in the chart $K_3$} 
\label{secK3}

We investigate the dynamics in the chart $K_3$~\eqref{K3} for $\lambda > 1$. 
First, recall from~\eqref{kappa23d} that the change of ccordinates $k_{23}: K_2 \to K_3$ is given by
\begin{equation*}
\epsilon_3 = x_2^{-2}, \quad  y_3 = x_2^{-1} y_2, \quad r_3 =  x_2 r_2, \quad h_3 =  x_2 h_2\,,
\end{equation*}
Symmetrically to the chart $K_1$, we define
$$ D_3 := \{(r_3, y_3, \epsilon_3, h_3) \in \mathbb{R}^4 \,:\, r_3 \in [ 0, \rho], 
\epsilon_3 \in [0, \delta], h_3 \in [ 0, \nu] \}$$
and
$$ \hat D_3 := \{(r_3, y_3, \epsilon_3, h_3) \in \mathbb{R}^4 \,:\, r_3 \in [ \rho/2, \rho], 
\epsilon_3 \in [\delta/4, \delta], h_3 \in [ \nu/2, \nu] \}\,.$$
Since we need to have $k_{23}\left( \Sigma_{2,\txte}^{\textnormal{out}} \right) \subset 
\Sigma_3^{\textnormal{in}}$, a suitable choice is given by
\begin{equation*}
\Sigma_{3}^{\textnormal{in}} :=  \{(r_3, y_3, \epsilon_3, h_3) \in D_3 \,:\, \left(\delta^{-1} 
+ 4 \nu \delta^{-1}\right)^{-1}  \leq \epsilon_3 \leq  \delta \}\,.
\end{equation*}
Furthermore, we will simply set
\begin{equation*}
\Sigma_{3}^{\textnormal{out}} :=  \{(r_3, y_3, \epsilon_3, h_3) \in D_3 \,:\, r_3 = \rho,  
h_3 = \nu, \epsilon_3 = \delta/4, \ y_3 > 0  \}\,,
\end{equation*}
and will end the analysis with the point of the trajectory which is closest to $\Sigma_{3}^{\textnormal{out}}$.
The dynamics, desingularised by choosing $h = h_3/r_3$, look as follows:
\begin{align} \label{K3discrete}
\tilde{r}_3 &= r_3(1 + h_3F_3(y_3, \epsilon_3)), \nonumber \\
\tilde y_3 &= (y_3 + \epsilon_3 h_3)(1 + h_3F_3(y_3, \epsilon_3))^{-1}\,\nonumber \\
\tilde \epsilon_3 &=  \epsilon_3 (1 + h_3F_3(y_3, \epsilon_3))^{-2}\,\nonumber \\
\tilde{h}_3 &= h_3(1 + h_3F_3(y_3, \epsilon_3)),
\end{align}
where $F_3(y_3, \epsilon_3) = 1 - y_3^2 + \lambda \epsilon_3$. For any $h_3$ system~\eqref{K3discrete} 
has the fixed points
$$ v_{\txtr,3}^- = (0,-1,0,h_3), \quad v_{\txtr,3}^+ = (0,1,0,h_3).$$
The points $v_{\txtr,3}^-$ and $v_{\txtr,3}^+$ have a three-dimensional centre eigenspace and a 
one-dimenional unstable eigenspace with the eigenvalue $1+2h_3$. Hence, unlike the analogous case in the 
chart $K_1$, the stability does not depend on the size of $h_3$. The most relevant manifold for our problem 
is given by
$$ W := \{w^{\textnormal{out}}(h_3) := (0,0,0,h_3)\;\:\;\ h_3 \in [0, \nu]\}\,,$$
which is a line for system~\eqref{K3discrete} within $D_3$. The points $w^{\textnormal{out}}(h_3)$ 
have two stable and two unstable eigenvalues 
$$\lambda_1 = (1 +h_3)^{-2}, \quad \lambda_2 = (1 + h_3)^{-1}, \quad \lambda_3 = 1+ h_3 , \quad \lambda_4 
= 1 + 2h_3\,,$$
such that the stability corresponds to the time-continuous problem independently from $h_3$. Note that the chart $K_3$ differs in that respect from the chart $K_1$ where preservation of stability is bound to the stability criteria of the Euler method known from the Dahlquist test equation.

The eigenvalues $\lambda_1, \lambda_2$ correspond with the $\epsilon_3$- and $y_3$-directions and 
$\lambda_3, \lambda_4$ with the $r_3$- and $h_3$-directions. We extend the set $W$ to the attracting invariant manifold $M_{\txta,3}$, which is given in $D_3$ by a graph $y_3 = l_{3}(\epsilon_3, h_3)$. One can derive 
$l_{3}$ from the discrete invariance equation
\begin{equation} 
\label{invarianceK3}
l_{3} (\tilde \epsilon_3, \tilde h_3)= \frac{l_{3} (\epsilon_3, h_3) + \epsilon_3 h_3}{1 + h_3F_3( 
l_{3} (\epsilon_3, h_3), \epsilon_3)}\,.
\end{equation} 
Note that, analogously to the continuous time case, there is the resonance $\lambda_1 \lambda_3 = \lambda_2$, 
which makes the description of the dynamics close to $W$ and $M_{\txta,3}$ a delicate problem. However, 
the exiting behaviour can still be estimated by a relatively simple argument without a full analysis
of the resonance as follows. Let $P_3$ denote the map given by~\eqref{K3discrete} and $\pi_y$ the 
projection to the $y$-component. 

\begin{proposition} 
\label{TranstionK3}
The transition map $\Pi_3$ from $\Sigma_3^{\textnormal{in}}$ to the vicinity of 
$\Sigma_3^{\textnormal{out}}$ given by 
$$\Pi_3(z) = P^{m^*(z)}(z)\,, \text{ where } m^*(z) = \argmin_{n \in \mathbb{N}} 
\dist( P_3^{n}(z), \Sigma_3^{\textnormal{out}})\,, \ z \in \Sigma_3^{\textnormal{in}}\,,$$
is well-defined on $k_{23}\left( \Sigma_{2,\txte}^{\textnormal{out}} \right) $. 
Furthermore, for $z \in k_{23}\left( \Sigma_{2,\txte}^{\textnormal{out}} \right) \subset 
\Sigma_3^{\textnormal{in}}$ we have $\pi_y(\Pi_3(z)) = \mathcal{O}(\delta^{1/3})$.
\end{proposition}

\begin{proof}
By the construction of $\Sigma_{2,\txte }^{\textnormal{out}}$, 
Proposition~\ref{propK2}, and the fact that $y_3 = x_2^{-1} y_2$ we have $\pi_y(z) = 
\mathcal{O}\left(\delta^{1/3}\right)$ for any $z \in k_{23}\left( \Sigma_{2,\txte}^{\textnormal{out}} \right)$. 
Further note that $F_3$ is clearly positive as long as $y_3$ maintains some positive order of $\delta$. 
Since $\epsilon_3 h_3 = \mathcal{O}(\delta \nu)$ in $D_3$, we can immediately infer that 
$$\pi_y(P_3^n(z)) = \mathcal{O}\left(\delta^{1/3}\right)\qquad \text{for all $n \leq m^*(z)$.}$$
This implies both statements as $F_3$ stays positive along the trajectory.
\end{proof}

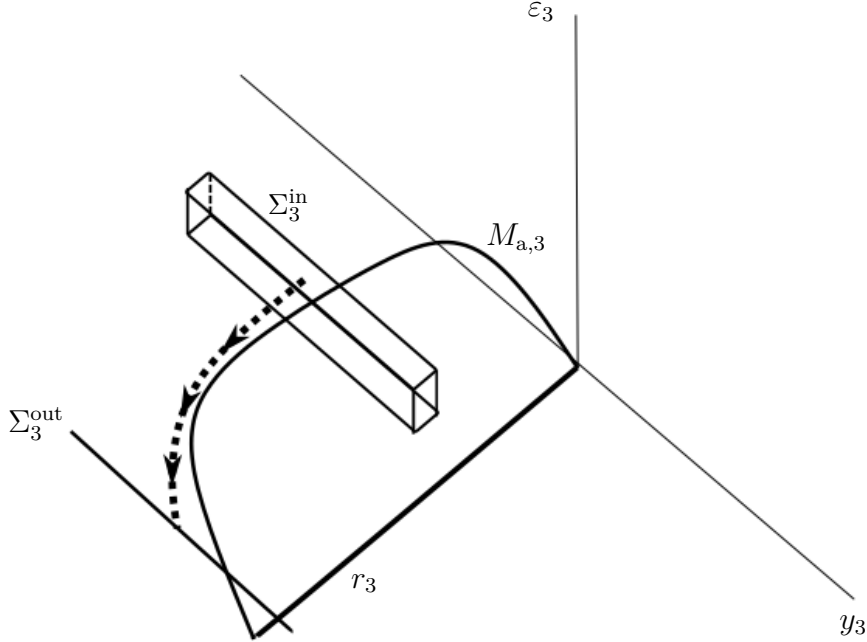
\begin{figure}[htbp]
        \centering
  		\begin{overpic}[width=0.8\textwidth]{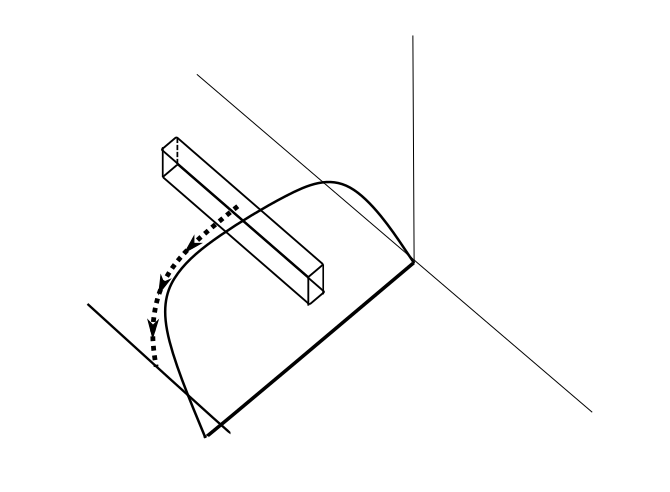}            
           \put(32,49){\small $\Sigma_{3}^{\textnormal{in}}$}
           \put(6,28){ $\Sigma_{3}^{\textnormal{out}}$}
           \put(86,9){ $y_3$} 
           \put(56,68){ $\epsilon_3$}
           \put(52,46){ $M_{\txta,3}$} 
           \put(39,13){ $r_3$}   
        \end{overpic}
		\caption{Illustration of the dynamics in chart $K_3$ in $(r_3,y_3,\epsilon_3)$-space. Note that, due to the relation $h_3 = r_3 h$, we can interpret $l_{3}(\epsilon_3, h_3)$ as a function of $(\epsilon_3, r_3)$ and sketch its graph, the invariant manifold $M_{\txta,3}$, as in the figure above. The figure shows a trajectory (dotted line) starting in $\Sigma_{3}^{\textnormal{in}}$ up to reaching the vicinity of $\Sigma_{3}^{\textnormal{out}}$ close to $M_{\txta,3}$.}
        \label{fig:K3}
\end{figure}

\subsection{Connecting the charts and proof of the Theorem} 
\label{secproof}

Finally, we can prove Theorem~\ref{transcritical_discrete} by combining the 
dynamics in $K_1$, $K_2$ and $K_3$ into a global picture.
\begin{proof}[Proof of Theorem~\ref{transcritical_discrete}]
We have proven the statements in charts $K_1, K_2, K_3$ for $\bar \epsilon = \delta$ with 
$\delta$ sufficiently small. Hence, we choose $\epsilon_0 = \rho^2 \delta_0$, where $\delta_0$ 
is the largest value of $\delta/4$ such that the statements hold. We did not use any further 
restrictions on $h$ apart from $\rho h < \delta$ and $\rho h < \frac{1}{8}$. Hence, it is enough 
to assume $ \rho^3 h < \epsilon$.

As before, we distinguish several cases. First, we consider $\lambda < 1$. We define the map 
$\bar \Pi_\txta$ from $ R_1 \subset \Sigma_{1,-}^{\textnormal{in}}$ to the vicinity of 
$\Sigma_{1,+}^{\textnormal{out}}$ by
$$ \bar \Pi_\txta := \Pi_{1,+} \circ k_{21} \circ \Pi_2 \circ k_{12} \circ \Pi_{1, -}\,,$$
where $\Pi_2: \Sigma_2^{\textnormal{in}} \to \Sigma_{2,\txta }^{\textnormal{out}}$ is the map 
well-defined by Proposition~\ref{propK2}. We have seen
that 
$$k_{12} \left(\Pi_{1,-} \left( R_1\right) \right) \subset \Sigma_2^{\textnormal{in}} 
\ \text{ and } \ k_{21} \left( \Sigma_{2,\txta }^{\textnormal{out}} \right) \subset R_2 
\subset \Sigma_{1, +}^{\textnormal{in}}\,.$$
Hence, $\bar \Pi_\txta$ is indeed a well-defined map. We have that $ \Pi_\txta = \Phi \circ 
\bar \Pi_\txta \circ \Phi^{-1}$, where $\Delta^{\textnormal{in}} = \Phi(R_1) $ and 
$\Delta_\txta^{\textnormal{out}} \subset \Phi \left(\Sigma_{1,+}^{\textnormal{out}}\right)$ 
is an interval about $S_\txta^+$ of the same size as $\Delta^{\textnormal{in}}$. We observe with 
Proposition~\ref{Invariance_Prop} that $ \Phi \left(M_{\txta,1}^- \right) \subset S_{\txta, \epsilon,h}^{-}$ 
and $ \Phi \left(M_{\txta,1}^+\right) \subset S_{\txta, \epsilon,h}^{+}$, and, by the choices of $R_1$ 
and $R_2$, that $\Delta^{\textnormal{in}} \cap S_{\txta,\epsilon,h}^-$ and $ \Pi_\txta \left( 
\Delta^{\textnormal{in}} \right) \cap S_{\txta,\epsilon,h}^+$ are nonempty. Summarizing, we can 
conclude that $\Pi_\txta$ maps $\Delta^{\textnormal{in}}$ including $\Delta^{\textnormal{in}} 
\cap S_{\txta,\epsilon,h}^-$ to a set about $S_{\txta, \epsilon,h}^{+}$. The distance between any 
point in $ \Pi_\txta \left( \Delta^{\textnormal{in}} \right)$ and $\Delta_\txta^{\textnormal{out}}$ 
is of order $\mathcal{O}(h \epsilon)$ since for $(x,y) 
\in \Delta_\txta^{\textnormal{out}} \cap S_{\txta,\epsilon,h}^+$ it is bounded by $h(x^2 - y^2 + \lambda \epsilon)$ due to the definition 
of $\Pi_\txta$ and we have by~\eqref{lplus} that
$$\left| x^2 - y^2 \right| = \left| x - y \right| \left| x + y \right| = 
\mathcal{O}(\delta \rho)\mathcal{O}( \rho) =  \mathcal{O}\left(\frac{\epsilon}{\rho} 
\right)\mathcal{O}( \rho)= \mathcal{O}(\epsilon)\,.$$
Furthermore, Proposition~\ref{contraction_K1} says that $\Pi_{1,-}|R_1$ and 
$\Pi_{1,+}|R_2$ are contractions in the $y_1$ direction with rates of order at least 
$ \mathcal{O} \left((1- c)^{\frac{C}{\rho h \delta} } \right)$ for some constant $C>0$. Since 
$\Pi_2$ is also contracting and due to $\mathcal{O}(\delta) =  \mathcal{O} 
\left(\frac{\epsilon}{\rho^2}  \right)$, we obtain that $\Pi_\txta 
\left(\Delta^{\textnormal{in}}\right)$ has $y$-width at most of order 
$\mathcal{O}\left( (1- c)^{\frac{C \rho }{h \epsilon}} \right)$.

Let now $\lambda > 1$.  
We define the map $\bar \Pi_\txte$ from $ R_1 \subset \Sigma_{1,-}^{\textnormal{in}}$ to the vicinity 
of $\Sigma_{3}^{\textnormal{out}}$ by
$$ \bar \Pi_\txte := \Pi_{3} \circ k_{23} \circ \Pi_2 \circ k_{12} \circ \Pi_{1, -}\,.$$ 
Again, we know that $k_{12} \left(\Pi_{1,-} \left( R_1\right) \right) \subset \Sigma_2^{\textnormal{in}}$, 
and furthermore from Proposition~\ref{TranstionK3} that $\Pi_3$ is well-defined on $k_{23} \left( 
\Sigma_{2,\txte }^{\textnormal{out}} \right) \subset \Sigma_{3}^{\textnormal{in}}$. Hence, 
$\bar \Pi_\txte$ is indeed a well-defined map. We have that $ \Pi_\txte = \Phi \circ \bar \Pi_\txte 
\circ \Phi^{-1}$, where again $\Delta^{\textnormal{in}} = \Phi(R_1) $ and $\Delta_\txte^{\textnormal{out}} 
\subset \Phi \left(\Sigma_{3}^{\textnormal{out}}\right)$ is an interval perpendicular to the $x$-axis. 
It follows immediately that $S_{\txta, \epsilon,h}^-$ passes through $\Delta_\txte^{\textnormal{out}}$ 
at a point $(\rho, k(\epsilon))$. Using Proposition~\ref{TranstionK3} we can characterize $k(\epsilon) 
 = \rho \mathcal{O}(\delta^{1/3}) = \rho^{1/3} \mathcal{O}(\epsilon^{1/3})$.  

The fact that $\Pi_\txte \left( \Delta^{\textnormal{in}} \right)$ has $y$-width 
$\mathcal{O}\left( (1- c)^{\frac{C \rho}{ h \epsilon}}\right)$ follows as for $\lambda < 1$.
The distance between any point in $ \Pi_\txte \left( \Delta^{\textnormal{in}} \right)$ and 
$\Delta_\txta^{\textnormal{out}}$ is of order $\mathcal{O}(h (\epsilon + \rho^2))$ since for $(x,y) 
\in \Delta_\txte^{\textnormal{out}} \cap S_{\txta,\epsilon,h}^-$ it is bounded by $h(x^2 - y^2 + 
\lambda \epsilon)$ due to the definition of $\Pi_\txte$ and $ x^2 - y^2 = 
\mathcal{O}\left(\rho^2 - \epsilon^{2/3}\rho^{2/3} \right)$.
This finishes the proof.
\end{proof}

\section{Summary and Outlook} 
\label{summaryoutlook}

We have applied the blow-up method to the Euler discretization of a fast-slow system 
with a transcritical singularity at the origin. We have shown that the qualitative behaviour 
of the slow manifolds is preserved by the discretization for any choice of $0 < h  < \epsilon$ (setting $\rho =1$), 
where $h$ denotes the time step size and $\epsilon$ the small time scaling parameter of the 
fast-slow system. The central part of the proof lies in the scaling chart $K_2$ of the manifold 
corresponding with the blown up singularity and is expressed in Proposition~\ref{propK2}. The 
proof of the proposition uses direct analysis of the map and, by that, can be seen of an alternative 
way of also showing the continuous analogue of the result when $h \to 0$. Furthermore, we are able 
to estimate transition times of trajectories by the analysis of the entering chart $K_1$ and give 
a bound for the $y$-component in the exiting chart $K_3$. In fact, our estimates provide a
very fine control on individual trajectories, which is a potential advantage of the discrete-time
framework for fast-slow systems.

We consider the work presented in this paper as one of the key steps towards a more comprehensive 
analysis of non-hyperbolic fixed points and non-hyperbolic submanifolds of fixed points in maps 
with multiple time scales. Whereas the normally hyperbolic theory for discrete-time multiple 
time scale systems is already quite well developed in \cite{HPS77, ns2013, Poe03}, the geometric 
desingularisation of non-hyperbolic objects for maps still needs several extensions. For example, 
our problem~\eqref{map_intro} is based on an explicit Euler discretization, which is obviously 
the most straight forward scheme. We conjecture that one can use the more direct blow-up approach
we use here for maps corresponding to ODEs also for other time-discretization schemes. 

There are several reasons, why particular schemes should be checked: it is well-known from
the area of geometric integration and the general theory of structure-preserving discretizations
\cite{HLW06} that only certain discrete-time schemes preserve relevant dynamical
properties, e.g. adiabatic invariants for the Hamiltonian systems case \cite{HLW06} or
certain asymptotic dynamics for the dissipative case \cite{Jin99}.
For multiple time scale maps, Runge-Kutta methods have been studied from a geometric 
viewpoint \cite{ns95, ns96}. It remains to clarify more systematically, for which discretization 
the geometric blow-up approach can be applied and what the relation between the two small
parameters $0 <h, \epsilon \ll 1$ must be. In this context, an interesting problem are canard 
explosions in discrete-time \cite{ER03, Fru88}, where also a third parameter is going to play a
key role. Working out this case starting from the
geometric approach by Krupa and Szmolyan for fast-slow ODEs \cite{ks2001/3} is currently 
work in progress by the authors of this paper. 

\newpage

\bibliographystyle{plain}
\bibliography{mybibfile}

\end{document}